\documentclass[a4paper,11pt]{article}

\usepackage{amsmath}
\usepackage{amsthm}
\usepackage{amsfonts}
\usepackage{amssymb}
\usepackage{t-angles}

\usepackage[all]{xy}
\usepackage{amscd}

\newcommand{\si}{\sigma}

\newcommand{\tht}{\theta}

\newcommand{\id}{\mathrm{id}}

\newcommand{\ot}{\otimes}

\newcommand{\trl}{\triangleleft}
\newcommand{\trr}{\triangleright}

\def\ppr{\rightharpoonup}
\def\ppl{\leftharpoonup}

\newcommand{\li}{{}_{1}}
\newcommand{\lii}{{}_{2}}

\newcommand{\lmo}{{}_{(0)}} 

\newcommand{\loo}{{}_{(0)}}
\newcommand{\loi}{{}_{(-1)}}

\newcommand{\lmoo}{{}_{(0)}}

\newcommand{\lmi}{{}_{(1)}}

\newcommand{\lmoi}{{}_{(-1)}}

\newcommand{\mo}{{}_{(0)}}
\newcommand{\mi}{{}_{(1)}}

\newcommand{\moi}{{}_{(-1)}}

\newcommand{\boo}{{}_{[0]}}
\newcommand{\bi}{{}_{[1]}}

\newcommand{\boi}{{}_{[-1]}}

\newcommand{\poo}{{}_{[0]}}

\newcommand{\ppi}{{}_{<1>}}
\newcommand{\pii}{{}_{<2>}}
\newcommand{\qi}{{}_{\{1\}}}
\newcommand{\qii}{{}_{\{2\}}}

\def\rbiprod{{\cdot\kern-.33em\triangleright\!\!\!<}}
\def\lbiprod{{>\!\!\!\triangleleft\kern-.33em\cdot\, }}

\def\lrbiprod{{\ \cdot\kern-.60em\triangleright\kern-.33em\triangleleft\kern-.33em\cdot\, }}

\def\lprod{{>\!\!\!\triangleleft\kern-.33em\ \, }}

\newcommand{\lrcoprod}{{\,\blacktriangleright\!\!\blacktriangleleft\, }}

\allowdisplaybreaks
\newtheorem{theorem}{Theorem}[section]
\newtheorem{lemma}[theorem]{Lemma}

\theoremstyle{definition}
\newtheorem{definition}[theorem]{Definition}
\newtheorem{remark}[theorem]{Remark}

\title{Braided anti-flexible bialgebras}
\author{Tao Zhang, Hui-Jun Yao}
\date{}

\begin{document}
 \maketitle

 \setcounter{section}{0}

\begin{abstract}
We introduce the concept of braided anti-flexible bialgebra  and construct cocycle bicrossproduct  anti-flexible bialgebras.
As an application, we solve the extending problem for anti-flexible bialgebras by using some non-abelian cohomology theory.
\par\smallskip
{\bf 2020 MSC:} 17A20, 17D25, 16T10

\par\smallskip
{\bf Keywords:}
Braided  anti-flexible bialgebras, cocycle bicrossproducts, extending structures,  non-abelian cohomology.
\end{abstract}


\section{Introduction}
As a special type of Lie-admissible algebras, anti-flexible algebras have been studied  by Anderson, Outcalt, Kosier and  Rodabaugh in \cite{AO,K,R}.
Very recently,  the concept of anti-flexible bialgebras is introduced by Dassoundo, Bai and Hounkonnou in \cite{DBH}.
The theory of Manin triples and anti-flexible Yang-Baxter equation for  anti-flexible algebras are developed in the same paper.
Pre-anti-flexible algebras and pre-anti-flexible bialgebras are studied by Dassoundo in \cite{D1,D2}.

On the other hand, the theory of extending structure for many types of algebras were well  developed by A. L. Agore and G. Militaru in \cite{AM1,AM2,AM3,AM4,AM5,AM6}.
Let $A$ be an algebra and $E$ a vector space containing $A$ as a subspace.
The extending problem is to describe and classify all algebra structures on $E$ such that $A$ is a subalgebra of $E$.
They show that associated to any extending structure of $A$ by a complement space $V$, there is an  unified product on the direct sum space  $E\cong A\oplus V$.
Recently, extending structures for Lie bialgebras, 3-Lie algebras,  infinitesimal bialgebras and Lie conformal superalgebras were studied  in \cite{Hong,Z2,Z3,Z4,ZCY}.

Since the extending structure for  anti-flexible algebras and the cohomology theory of anti-flexible bialgebras have not been developed in the literature.
The aim of this paper is to fill in these gaps. The motivation is from quantum group theory. In \cite{Ra85,Dr86,BD99,BD01}, the concept of braided Hopf algebras was provided and the construction of cross product bialgebras was studied in detail. See also \cite{Ma90a,Ma95,Mas00,Zh99}.
In \cite{So96,Ma00,Z1}, the notion of  braided Lie bialgebras  was introduced and the construction of cocycle bicrossproducts Lie bialgebras was developed.
It is a natural question whether there are similar constructions in theory of anti-flexible bialgebras.

In this paper, we provided the concept of braided anti-flexible bialgebras.
It is showed that this new concept will play a key role in considering extending problem for anti-flexible bialgebras.
Secondly, the theory of unified product for anti-flexible bialgebras is also developed and the construction of cocycle bicrossproduct anti-flexible bialgebras is given.
Finally, we solve the extending problem for anti-flexible bialgebras by using some non-abelian cohomology theory.

This paper is organized as follows. In Section 2, we recalled some definitions and fixed some notations about anti-flexible algebras. In Section 3, we introduced the concept of braided anti-flexible bialgebras and proved the bosonisation theorem associating braided anti-flexible bialgebras to ordinary anti-flexible bialgebras.
In Section 4, we defined the notion of matched pairs of  braided anti-flexible bialgebras.
Besides, we constructed cocycle bicrossproduct anti-flexible bialgebras through two generalized braided anti-flexible bialgebras.
In Section 5, we studied the extending problems for anti-flexible bialgebras and proved that they can be classified by some non-abelian cohomology theory.

Throughout the following of this paper, all vector spaces will be over a fixed field of character zero.
An  algebra or  a   coalgebra   is  denoted by $(A, \cdot)$ or $(A, \Delta)$.
The identity map of a vector space $V$ is denoted by $\id_V: V\to V$ or simply $\id: V\to V$.
The flip map $\tau: V\ot V\to V\ot V$ is defined by $\tau(u\ot v)=v\ot u$ for all $u, v\in V$.

\section{Preliminaries}

\begin{definition}Let $A$ be a vector space equipped with a multiplication $\cdot: A\otimes A\rightarrow A$.  Then $A$ is called an anti-flexible algebra if the following anti-flexible identity is satisfied:
\begin{equation}
(a,~ b,~ c)=(c,~ b,~ a),
 \end{equation}
  or  equivalently,
\begin{equation}
 (a\cdot b)\cdot c-a\cdot (b\cdot c)=(c\cdot b)\cdot a-c\cdot (b\cdot a),
 \end{equation}
 where $a, b, c \in A$ and  the associator is denoted by $(a,~b,~c)=(a\cdot b)\cdot c-a\cdot (b\cdot c).$
In the following, we always omit $ ``\cdot" $ and write the multiplication by $ab$ for simplicity.
\end{definition}

\begin{definition}An  anti-flexible coalgebra $A$ is a vector space equipped with a comultiplication $\Delta: A\rightarrow A\otimes A$ such that the following anti-flexible condition is satisfied,
\begin{equation}
(\Delta \otimes \id)\Delta(a)-( \id\otimes\Delta)\Delta (a)=\tau_{13}\left( (\Delta \otimes \id)\Delta(a)-( \id\otimes\Delta)\Delta (a)\right),
 \end{equation}
where $\tau_{13}(a\ot b\ot c)=c\ot b\ot a$.
We denote an anti-flexible coalgebra by $(A, \Delta)$.
\end{definition}

\begin{definition} \label{dfnlb} (\cite{DBH}) An  anti-flexible bialgebra $A$ is a vector space equipped simultaneously with an anti-flexible algebra structure $(A, \cdot)$  and an anti-flexible  coalgebra structure $(A, \Delta)$ such that the following
compatibility conditions are satisfied,
\begin{equation}\label{eq:LB0}
 \Delta(ab)+\tau\Delta(ba)=\sum  a\li b\ot a\lii +b a\lii\ot a\li +b\li \ot ab\lii+b\lii\ot b\li a ,
 \end{equation}
 \begin{equation}\label{eq:LB1}
  (\id-\tau )\Big( a\li\ot a\lii b+ab\li\ot b\lii-ba\li\ot a\lii-b\li\ot b\lii a \Big)=0,
 \end{equation}
where we  use the sigma notation $\Delta(a):=\sum a\li\ot a\lii$.
We denote an  anti-flexible bialgebra by $(A, \cdot, \Delta)$.
\end{definition}

\begin{remark} The anti-flexible  bialgebras in the above Definition \ref{dfnlb} is same as in  \cite[Definition 3.4]{DBH}, but we use different notations as \cite{DBH}, in which  \eqref{eq:LB0} and \eqref{eq:LB1} were written as
\begin{align}\label{eq:LB00}
\Delta(a b)+\tau\Delta(ba)=&(\tau(\id \otimes L(b))+ R(b)\otimes \id)\Delta(a)+(\tau(R(a)\otimes \id) + \id\otimes L(a))\Delta(b),\\
 \nonumber(\tau(\id\otimes R(b))&-\id\otimes R(b)- \tau(L(b)\otimes \id)+L(b)\otimes \id)\Delta(a)=\\
 =&(\tau(\id\otimes R(a))-\id\otimes R(a)- \tau(L(a)\otimes \id)+L(a)\otimes \id)\Delta(b),\label{eq:LB10}
\end{align}
where $L(a)$ and $R(a)$ denote the left and right multiplication operators respectively.
It is easy to see that \eqref{eq:LB00} and  \eqref{eq:LB10} are equivalent to  \eqref{eq:LB0} and  \eqref{eq:LB1} using the sigma notation.
For simplicity, we also would like to denote
\begin{align*}
&\Delta (a)\cdot b:= \sum a\li\ot a\lii b=(\id\otimes R(b))\Delta (a),\\
&a\cdot \Delta(b):=\sum  a b\li\ot b\lii=(L(a)\otimes \id)\Delta(b),\\
&\Delta(a)\bullet b:=\sum a\li b\ot a\lii=(R(b)\otimes \id)\Delta (a),\\
& a\bullet \Delta(b):=\sum  b\li\ot a b\lii=(\id\otimes L(a))\Delta(b).
\end{align*}
 Thus we also write the compatibility conditions \eqref{eq:LB00} and  \eqref{eq:LB10} as
\begin{equation}\label{eq:LB2}
 \Delta(ab)+\tau\Delta(ba)=\sum \Delta(a)\bullet b +b\cdot \tau\Delta(a)+a\bullet \Delta(b)+\tau\Delta(b)\cdot a,
 \end{equation}
\begin{equation}\label{eq:LB3}
  (\id-\tau )\Big( \Delta (a)\cdot b+a\cdot \Delta(b) -b\cdot \Delta(a)-\Delta (b)\cdot a\Big)=0.
 \end{equation}
 \end{remark}

\begin{definition}
Let ${A}$ be an anti-flexible algebra and $V$ be a vector space. Then $V$ is called an ${A}$-bimodule if there is a pair of linear maps $ \trr: {A}\otimes V \to V, ({a}, v) \to {a} \trr v$ and $\trl: V\otimes {A} \to V, (v, {a}) \to v \trl {a}$  such that the following conditions hold:
\begin{eqnarray}
&&  (ab) \trr v- a\trr ({b}\trr v)=(v \trl {b}) \trl a-v \trl ({b}a), \\
&&  (a\trr v)\trl {b}-a \trr (v\trl {b}) =({b}\trr v)\trl a-{b} \trr (v\trl a),
\end{eqnarray}
for all $a, b\in {A}$ and $v\in V.$

\end{definition}
The category of  bimodules over $A$ is denoted  by ${}_{A}\mathcal{M}{}_{A}$.

\begin{definition}
Let ${A}$ be an anti-flexible coalgebra,  $V$ a vector space.  Then $V$ is called an ${A}$-bicomodule if there is a pair of linear maps $\phi: V\to {A}\otimes V$ and $\psi: V\to V\otimes {A}$  such that the following conditions hold:
\begin{eqnarray}
 &&\left(\Delta_{A} \otimes \id _{V}\right)\phi(v)-\left(\id _{A} \otimes \phi\right) \phi(v)=\tau_{13}\Big(\left( \psi\otimes \id _{A} \right) \psi(v)-\left(\id_{V}\ot\Delta_{A} \right)\psi(v)\Big),\\
 &&(\phi\ot \id_{A})\psi(v)-(\id_{A} \ot \psi)\phi(v)=\tau_{13}\Big((\phi\ot \id_{A})\psi(v)-(\id_{A} \ot \psi)\phi(v)\Big).
\end{eqnarray}
If we denote by  $\phi(v)=v\moi\ot v\mo$ and $\psi(v)=v\mo\ot v\mi$, then the above equations can be written as
\begin{eqnarray}
  &&\Delta_{A}\left(v_{(-1)}\right) \otimes v_{(0)}-v_{(-1)} \otimes \phi\left(v_{(0)}\right)=\tau_{13}\Big(\psi\left(v_{(0)}\right) \otimes v_{(1)}-v_{(0)} \otimes \Delta_{A}\left(v_{(1)}\right)\Big),\\
  &&\phi(v_{(0)})\ot v_{(1)}-v_{(-1)}\ot \psi(v_{(0)})=\tau_{13}\Big(\phi(v_{(0)})\ot v_{(1)}-v_{(-1)}\ot \psi(v_{(0)})\Big ).
\end{eqnarray}
\end{definition}
The category of  bicomodules over $A$ is denoted by ${}^{A}\mathcal{M}{}^{A}$.

\begin{definition}
Let ${H}$ and  ${A}$ be anti-flexible algebras. An action of ${H}$ on ${A}$ is a pair of linear maps $\trr:  {H}\otimes {A} \to {A}, (x, a) \to x \trr a$ and $\trl: {A}\otimes {H} \to {A}, (a, x) \to a \trl x$  such that $A$ is an $H$-bimodule and  the following conditions hold:
\begin{eqnarray}
  &&(x\trr a)b -x\trr (ab)=(ba)\trl x-b(a \trl x),\\
  &&a(x\trr b)-(a\trl x)b=b(x\trr a )-(b\trl x)a,
\end{eqnarray}
for all $x\in {H}$ and $a, b\in {A}$.  In this case, we call $(A, \, \trr, \, \trl)$ to be an $H$-bimodule algebra.
\end{definition}

\begin{definition}
Let ${H}$ and  ${A}$ be anti-flexible coalgebras. An coaction of ${H}$ on ${A}$ is a pair of linear maps $\phi: {A}\to {H}\otimes {A}$ and $\psi: {A}\to {A}\otimes {H}$ such that $A$ is an $H$-bicomodule and  the following conditions hold:
\begin{eqnarray}
  &&(\phi\otimes\id_A) \Delta_{A}(a)-(\id_H \otimes \Delta_{A})\phi(a)=\tau_{13}\Big((\Delta_{A}\otimes \id_H)\psi(a)-(\id_A\otimes \psi) \Delta_{A}(a)\Big),\\
  &&(\psi\otimes\id_A)\Delta_{A}(a)-(\id_A\otimes \phi) \Delta_{A}(a)
  =\tau_{13}\Big((\psi\otimes\id_A)\Delta_{A}(a)-(\id_A\otimes \phi) \Delta_{A}(a)\Big).
\end{eqnarray}
If we denote by  $\phi(a)=a\moi\ot a\mo$ and $\psi(a)=a\mo\ot a\mi$, then the above equations can be written as
\begin{eqnarray}
    &&\phi\left(a_{1}\right) \otimes a_{2}-a_{(-1)} \otimes \Delta_{A}\left(a_{(0)}\right)=\tau_{13}\Big( \Delta_{A}\left(a_{(0)}\right) \otimes a_{(1)}-a_{1} \otimes \psi\left(a_{2}\right)\Big), \\
    &&\psi(a_{1})\ot a_{2}-a_{1}\ot \phi(a_{2})=\tau_{13}(\psi(a_{1})\ot a_{2}-a_{1}\ot \phi(a_{2})).
\end{eqnarray}
for all $a\in {A}$. In this case, we call $(A, \, \phi, \, \psi)$ to be an $H$-bicomodule  coalgebra.
\end{definition}

\begin{definition}
Let $({A},\cdot)$ be a given   anti-flexible  algebra (anti-flexible coalgebra,  anti-flexible   bialgebra), $E$ a vector space.
An extending system of ${A}$ through $V$ is an anti-flexible  algebra  (anti-flexible coalgebra,   anti-flexible  bialgebra) on $E$
such that $V$ is a complement subspace of ${A}$ in $E$, the canonical injection map $i: A\to E, a\mapsto (a, 0)$  or the canonical projection map $p: E\to A, (a, x)\mapsto a$ is an anti-flexible  algebra (anti-flexible coalgebra,    anti-flexible bialgebra) homomorphism.
The extending problem is to describe and classify up to an isomorphism  the set of all  anti-flexible   algebra  (anti-flexible coalgebra,   anti-flexible  bialgebra) structures that can be defined on $E$.
\end{definition}

We remark that our definition of extending system of ${A}$ through $V$ contains not only extending structure in \cite{AM1,AM2,AM3}
but also the global extension structure in \cite{AM5}.
In fact, the canonical injection map $i: A\to E$ is an anti-flexible  (co)algebra homomorphism if and only if $A$ is an   anti-flexible  sub(co)algebra of $E$.

\begin{definition}
Let ${A} $ be an   anti-flexible   algebra (anti-flexible coalgebra, anti-flexible    bialgebra), $E$  be an anti-flexible  algebra  (anti-flexible coalgebra,   anti-flexible bialgebra) such that
${A} $ is a subspace of $E$ and $V$ a complement of
${A} $ in $E$.  For a linear map $\varphi: E \to E$ we
consider the diagram:
\begin{equation}\label{eq:ext1}
\xymatrix{
   0  \ar[r]^{} &A \ar[d]_{\id_A} \ar[r]^{i} & E \ar[d]_{\varphi} \ar[r]^{\pi} &V \ar[d]_{\id_V} \ar[r]^{} & 0 \\
   0 \ar[r]^{} & A \ar[r]^{i'} & {E} \ar[r]^{\pi'} & V \ar[r]^{} & 0
   }
\end{equation}
where $\pi, \pi': E\to V$ are the projection maps and $i, i': A\to E$ are the inclusion maps.
We say that $\varphi: E \to E$ \emph{stabilizes} ${A}$ if the left square of the diagram \eqref{eq:ext1} is  commutative.

Let $(E, \cdot)$ and $(E, \cdot')$ be two   anti-flexible  algebra (anti-flexible coalgebra,   anti-flexible  bialgebra) structures on $E$.  $(E, \cdot)$ and $(E, \cdot')$ are called \emph{equivalent}, and we denote this by $(E, \cdot) \equiv (E, \cdot')$, if there exists an anti-flexible  algebra (anti-flexible coalgebra,  anti-flexible   bialgebra) isomorphism $\varphi: (E, \cdot)
\to (E, \cdot')$ which stabilizes ${A}$.  Denote by $Extd(E,{A} )$ ($CExtd(E,{A} )$, $BExtd(E,{A} )$) the set of equivalent classes of   anti-flexible   algebra (anti-flexible coalgebra,  anti-flexible   bialgebra) structures on $E$.
\end{definition}

\section{Braided anti-flexible bialgebras}
In this section, we introduce  the concepts of   anti-flexible Hopf bimodule and braided  anti-flexible bialgebra which are key concepts in the following sections.

\subsection{Anti-flexible Hopf bimodules and braided   anti-flexible bialgebras}
\begin{definition}
 Let $H$ be an   anti-flexible   bialgebra. An anti-flexible Hopf bimodule over $H$ is a space $V$ endowed with maps
\begin{align*}
&\trr: H\otimes V \to V,\quad \trl: V\otimes H \to V,\\
&\phi:V \to H \otimes V,\quad  \psi: V \to V\otimes H,
\end{align*}
such that $V$ is simultaneously an $H$-bimodule, an  $H$-bicomodule and
 the following compatibility conditions hold:
 \begin{enumerate}
\item[(HM1)]$\phi(x \trr v)+\tau\psi(v \trl x)=v_{(-1)}\ot (x\trr v_{(0)})+v_{(1)}\ot (v_{(0)}\trl x)$,
\item[(HM2)]$\psi(x \trr v)+\tau\phi(v\trl x)=\left(x\li \trr  v\right) \otimes x\lii+v_{(0)}\otimes xv_{(1)}+v_{(0)}\otimes v_{(-1)}x+(v\trl x_{2})\ot x_{1}$,

\item[(HM3)] $ (x\trr v\loo)\ot v_{(1)}-(v\trl x\li)\ot x\lii-v\loo\ot v_{(1)} x  $\\
$=\tau(xv\loi\ot v\loo+x\li\ot(x\lii\trr v)-v\loi\ot (v\loo\trl x))$.
\end{enumerate}
\end{definition}
We denote  the  category of  anti-flexible Hopf bimodules over $H$ by ${}^{H}_{H}\mathcal{M}{}^{H}_{H}$.

\begin{definition} Let $H$  be an  anti-flexible bialgebra.
Let $A$ be an anti-flexible algebra and an anti-flexible coalgebra in ${}^{H}_{H}\mathcal{M}{}^{H}_{H}$, we call $A$ a \emph{braided anti-flexible bialgebra},
if the following braided compatibility conditions are satisfied:
\begin{enumerate}
\item[(BB1)]
$\Delta_{A}(ab)+\tau\Delta_{A}(ba)$\\
$=a\li b\ot a\lii +b a\lii\ot a\li +b\li \ot ab\lii+b\lii\ot b\li a$\\
$+(a\loi\trr b)\ot a\loo+(b\trl a_{(1)})\ot a\loo+b\loo\ot (a\trl b_{(1)})+b\loo\ot (b\loi \trr a)$,
\item[(BB2)]
$(\id-\tau)\Big(a\li\ot a\lii b-ba\li\ot a\lii-b\li\ot b\lii a+ab\li\ot b\lii  \Big)$\\
 $+(\id-\tau)\Big(a\loo\ot(a_{(1)}\trr b)-(b\trl a\loi)\ot a\loo-b\loo\ot(b_{(1)}\trr a)+(a\trl b\loi)\ot b\loo \Big)=0$.
\end{enumerate}
\end{definition}

Here  $A$ is an anti-flexible algebra and an anti-flexible  coalgebra in ${}^{H}_{H}\mathcal{M}{}^{H}_{H}$ means that $A$ is simultaneously an $H$-bimodule anti-flexible  algebra (anti-flexible coalgebra) and $H$-bicomodule  anti-flexible algebra (anti-flexible coalgebra).


Now we construct anti-flexible bialgebras from braided anti-flexible bialgebras.
Let $H$  be an   anti-flexible bialgebra, $A$ be an anti-flexible  algebra and an  anti-flexible coalgebra in ${}^{H}_{H}\mathcal{M}{}^{H}_{H}$.
We define multiplication and comultiplication on the direct sum vector space $E:=A \oplus H$ by
$$
\begin{aligned}
&(a, x)(b, y):=(a b+x\trr b+a \trl y, x y), \\
&\Delta_{E}(a, x):=\Delta_{A}(a)+\phi(a)+\psi(a)+\Delta_{H}(x).
\end{aligned}
$$
This is called biproduct of ${A}$ and ${H}$ which will be  denoted   by $A\lbiprod H$.

\begin{theorem} Let  $H$ be an   anti-flexible bialgebra.
Then the biproduct $A\lbiprod H$ forms an   anti-flexible bialgebra if and only if  $A$ is an braided    anti-flexible bialgebra in ${}^{H}_{H}\mathcal{M}{}^{H}_{H}$.
\end{theorem}

\begin{proof}
It is easy to prove the multiplication and the comultiplication  are  anti-flexible.
Next, we show the first compatibility condition:
 $$\begin{aligned}
 &\Delta_{E}((a,x)(b,y))+\tau\Delta_{E}((b,y)(a,x))\\
 =&\Delta_{E}((a,x))\bullet (b,y)+(b,y)\cdot \tau\Delta_{E}((a,x))+(a,x)\bullet \Delta_{E}((b,y))+\tau\Delta_{E}((b,y))\cdot (a,x).
 \end{aligned}
$$
By direct computations, the left hand side is equal to
$$\begin{aligned}
&\Delta_{E}((a,x)(b,y))+\tau\Delta_{E}((b,y)(a,x))\\
=&\Delta_E(a b+x \trr b+a \trl y,  x y)+\tau\Delta_E(ba+y \trr a+b \trl x,  y x)\\
=&\Delta_A(a b)+\phi(a b)+\psi(a b)+\Delta_A(x \trr b)+\phi(x \trr b)+\psi(x \trr b)\\
&+\Delta_A(a \trl y)+\phi(a \trl y)+\psi(a \trl y)+\Delta_{H}(x y)\\
&+\tau\Delta_A(b a)+\tau\phi(b a)+\tau\psi(b a)+\tau\Delta_A(y \trr a)+\tau\phi(y \trr a)+\tau\psi(y \trr a)\\
&+\tau\Delta_A(b \trl x)+\tau\phi(b \trl x)+\tau\psi(b \trl x)+\tau\Delta_{H}(y x),
\end{aligned}
$$
and the right hand side is equal to
\begin{eqnarray*}
&& \Delta_{E}((a,x))\bullet (b,y)+(b,y)\cdot \tau\Delta_{E}((a,x))+(a,x)\bullet \Delta_{E}((b,y))+\tau\Delta_{E}((b,y))\cdot (a,x)\\
&=&\left(a_{1} \otimes a_{2}+a_{(-1)} \otimes a_{(0)}+a_{(0)} \otimes a_{(1)}+x_{1} \otimes x_{2}\right) \bullet(b, y)\\
&&+(b, y)\cdot\left( a_{2}\otimes a_{1}+ a_{(0)}\otimes a_{(-1)}+ a_{(1)}\otimes a_{(0)}+ x_{2}\otimes x_{1} \right)\\
&&+(a, x) \bullet\left(b_{1} \otimes b_{2}+b_{(-1)} \otimes b_{(0)}+b_{(0)} \otimes b_{(1)}+y_{1} \otimes y_{2}\right)\\
&&+\left( b_{2}\otimes b_{1}+ b_{(0)}\otimes b_{(-1)}+ b_{(1)}\otimes b_{(0)}+ y_{2}\otimes y_{1} \right)\cdot(a, x)\\
&=&a\li b\ot a\lii+(a\li \trl y)\ot a\lii+(a\loi\trr b)\ot a\loo+a\loi y\ot a\loo+a\loo b\ot a_{(1)}\\
&&+(a\loo \trl y)\ot a_{(1)}+x\li y\ot x\lii+(x\li \trr b)\ot  x\lii+\left(b a_{2}+y \trr a_{2}\right) \otimes a_{1}\\
&&+(b \trl a_{(1)}) \otimes a_{(0)}+y a_{(1)} \otimes a_{(0)}+\left(b a_{(0)}+y\trr a_{(0)}\right) \otimes a_{(-1)}\\
&&+\left(b \trl x_{2}\right) \otimes x_{1}+y x_{2}\otimes x_{1}+b\li\ot ab\lii+b\li\ot(x\trr b\lii)+b\loi\ot ab\loo\\
&&+b\loi\ot (x\trr b\loo)+b\loo\ot xb_{(1)}+b\loo\ot (a\trl b_{(1)})+y\li\ot xy\lii\\
&&+y\li\ot (a\trl y\lii)+b\lii \ot b\li a+b\lii \ot (b\li \trl x)+b\loo\ot b\loi x+b\loo\ot (b\loi\trr a)\\
&&+b_{(1)}\ot b\loo a+b_{(1)}\ot (b\loo \trl x)+y\lii\ot y\li x+y\lii\ot (y\li \trr a).
\end{eqnarray*}
Then the two sides are equal to each other if and only if

(1)$\Delta_{A}(ab)+\tau\Delta_{A}(ba)$

$\quad=a\li b\ot a\lii +b a\lii\ot a\li +b\li \ot ab\lii+b\lii\ot b\li a$

$\quad+(a\loi\trr b)\ot a\loo+(b\trl a_{(1)})\ot a\loo+b\loo\ot (a\trl b_{(1)})+b\loo\ot (b\loi \trr a)$,

(2) $\phi(x \trr b)+\tau\psi(b \trl x)=b_{(-1)}\ot (x\trr b_{(0)})+b_{(1)}\ot (b_{(0)}\trl x)$,

(3) $\psi(x \trr b)+\tau\phi(b\trl x)=\left(x\li \trr  b\right) \otimes x\lii+b_{(0)}\otimes xb_{(1)}+b_{(0)}\otimes b_{(-1)}x+(b\trl x_{2})\ot x_{1}$,

(4) $\phi(a b)+\tau\psi(ba)= b_{(-1)}\ot a b_{(0)}+b_{(1)} \otimes b_{(0)}a$,

(5) $\Delta_{A}(x\trr b)+\tau\Delta_{A}(b\trl x)= b_{2}\ot ( b_{1}\trl x)+ b_{1}\ot( x\trr  b_{2})$.


Finally, we show the second compatibility condition:
\begin{equation*}
  (\id-\tau )\Big( \Delta_{E} ((a,x))\cdot (b,y)-(b,y)\cdot \Delta_{E}((a,x))-\Delta_{E} ((b,y))\cdot (a,x)+(a,x)\cdot \Delta_{E}((b,y)) \Big)=0. 
 \end{equation*}
By direct computation, we have
\begin{eqnarray*}
&&\Delta_{E} ((a,x)\cdot (b,y)-(b,y)\cdot \Delta_{E}(a,x)-\Delta_{E} (b,y)\cdot (a,x)+(a,x)\cdot \Delta_{E}(b,y))\\
&=&\left(a_{1} \otimes a_{2}+a_{(-1)} \otimes a_{(0)}+a_{(0)} \otimes a_{(1)}+x_{1} \otimes x_{2}\right) \cdot(b, y)\\
&&-(b,y)\cdot\left(a_{1} \otimes a_{2}+a_{(-1)} \otimes a_{(0)}+a_{(0)} \otimes a_{(1)}+x_{1} \otimes x_{2}\right)\\
&&-\left(b_{1} \otimes b_{2}+b_{(-1)} \otimes b_{(0)}+b_{(0)} \otimes b_{(1)}+y_{1} \otimes y_{2}\right)\cdot(a,x)\\
&&+(a,x)\cdot\left(b_{1} \otimes b_{2}+b_{(-1)} \otimes b_{(0)}+b_{(0)} \otimes b_{(1)}+y_{1} \otimes y_{2}\right)\\
&=&a_{1} \otimes\left(a_{2} b+a_{2} \trl y\right)+a_{(-1)} \otimes\left(a_{(0)} b+a_{(0)} \trl y\right)+a_{(0)} \otimes\left(a_{(1)} \trr b\right)\\
&&+a_{(0)} \otimes\left(a_{(1)} y\right)+x_{1} \otimes\left(x_{2} \trr b\right)+x_{1} \otimes x_{2} y-(ba\li+y\trr a\li)\ot a\lii\\
&&-(b\trl a_{(-1)}+ya_{(-1)})\ot a\loo-ba\loo \ot a_{(1)}-(y\trr a\loo)\ot a_{(1)}\\
&&-(b\trl x\li)\ot x\lii-yx\li\ot x\lii-b\li\ot (b\lii a+b\lii\trl x)-b\loi\ot(b\loo a)\\
&&-b\loi\ot (b\loo \trl x)-b\loo\ot(b_{(1)}\trr a+b_{(1)}x)-y\li\ot (y\lii\trr a)-y\li\ot y\lii x\\
&&+\left(a b_{1}+x \trr b_{1}\right) \otimes b_{2}+(a \trl b_{(-1)}) \otimes b_{(0)}+\left(x b_{(-1)}\right) \otimes b_{(0)}\\
&&+\left(a b_{(0)}+x\trr b_{(0)}\right) \otimes b_{(1)}+\left(a \trl y_{1}\right) \otimes y_{2}+\left(x y_{1}\right) \otimes y_{2}.
\end{eqnarray*}
Thus the second compatibility condition holds if and only if

 (6) $(\id-\tau)\Big(a\li\ot a\lii b-ba\li\ot a\lii-b\li\ot b\lii a+ab\li\ot b\lii  \Big)$

 $\quad+(\id-\tau)\Big(a\loo\ot(a_{(1)}\trr b)-(b\trl a\loi)\ot a\loo-b\loo\ot(b_{(1)}\trr a)+(a\trl b\loi)\ot b\loo)  \Big)=0$,

(7) $a\loi\ot a\loo b-b\loi \ot b\loo a=\tau(ab\loo\ot b_{(1)}-ba\loo\ot a_{(1)})$,

(8)$(\id-\tau)((x\trr b\li)\ot b\lii)-b\li\ot(b\lii \trl x))=0 $,

(9) $ (x\trr b\loo)\ot b_{(1)}-(b\trl x\li)\ot x\lii-b\loo\ot b_{(1)} x  $

$\quad=\tau(xb\loi\ot b\loo+x\li\ot(x\lii\trr b)-b\loi\ot (b\loo\trl x))$.

Combine with all the conditions above,  we have found  that (2)--(3) and (9) are the conditions for $A$ to be an anti-flexible Hopf bimodule; and (4)--(5) with (7)--(8) are the conditions for $A$  to be an anti-flexible algebra and anti-flexible coalgebra in ${}^{H}_{H}\mathcal{M}{}^{H}_{H}$; finally,  (1) and (6) are the conditions for $A$ to be a braided anti-flexible bialgebra.

The proof is completed.
\end{proof}

\subsection{From quasitriangular anti-flexible bialgebra to braided anti-flexible bialgebra}

Let $(A,\cdot)$ be an anti-flexible algebra and
$\displaystyle  {r}=\sum_i{u_i\otimes v_i}\in A\otimes A$. Set
\begin{equation}
 {r}_{12}=\sum_iu_i\otimes v_i\otimes 1,\quad
 {r}_{13}=\sum_{i}u_i\otimes 1\otimes v_i,\quad r_{23}=\sum_i 1\otimes u_i\otimes v_i,
\end{equation}
In this section, we consider a special class of anti-flexible bialgebras,
that is, the anti-flexible bialgebra $(A,\Delta_r)$ on an anti-flexible algebra $(A,\cdot)$, with the map $\Delta_r$ defined by
\begin{equation}\label{delta-r}
\Delta_r(a)=\sum_i u_i\otimes av_i+v_ia\otimes  u_i.
\end{equation}

\begin{lemma}(\cite{DBH})
Let $(A, \cdot )$ be an anti-flexible algebra and $ {r}\in A\otimes A$. Let $\Delta_{r}:A\rightarrow A\otimes A$
be a map defined by \eqref{delta-r}.
If in addition, $ {r}$ is  skew-symmetric and $ {r}$ satisfies
\begin{equation}\label{eq_YBE}
 {r}_{{12}} {r}_{{13}}- {r}_{{23}} {r}_{{12}}+ {r}_{{13}} {r}_{{23}}=0,
\end{equation}
which is called the anti-flexible Yang-Baxter equation (AFYB).
Then $(A, \Delta_{r})$ is an anti-flexible bialgebra.
\end{lemma}

This kind of anti-flexible bialgebra is called a quasitriangular anti-flexible bialgebra.

\begin{theorem}\label{r-braided1}
 Let $(A, \cdot, r)$ be a quasitriangular anti-flexible bialgebra and $M$ an $A$-bimodule. Then $M$ becomes an anti-flexible Hopf bimodule  over $A$ with maps $\phi: M \rightarrow A \otimes M$ and $\psi: M \rightarrow M \otimes A$ given by
\begin{equation}
\phi(m):=\sum_{i} u_{i} \otimes m \trl v_{i},\quad \psi(m):= \sum_{i}v_{i}\trr m \otimes u_{i}
\end{equation}
\end{theorem}

\begin{proof}
We first prove that $M$ is a $A$-bicomodule:
$$
\left(\Delta_{r} \otimes \id\right)\phi(m)-\left(\id \otimes \phi\right) \phi(m)
=\tau_{13}\Big(\left( \psi\otimes \id \right) \psi(m)-\left(\id\ot\Delta_{r} \right)\psi(m)\Big).
$$
In fact, the left hand side is equal to
$$
\begin{aligned}
&\left(\Delta_{r} \otimes \id\right)\phi(m)-\left(\id \otimes \phi\right) \phi(m) \\
=&\Delta_{r}(u_{i})\ot m\trl v_{i}-u_{i}\ot \phi(m\trl v_{i})\\
=&u_{j}\ot u_{i}v_{j}\ot m\trl v_{i} +v_{j}u_{i}\ot u_{j}\ot m\trl v_{i}-u_{i}\ot u_{j}\ot (m\trl v_{i})\trl v_{j}\\
=&u_{i}\ot u_{j}\ot m \trl (v_{i}v_{j})-u_{i}\ot u_{j}\ot (m\trl v_{i})\trl v_{j},
\end{aligned}
$$
where the last equal holds by the (AFYB) and the antisymmetry of $r$, then right hand side is equal to
$$
\begin{aligned}
&\tau_{13}\Big(\left( \psi\otimes \id \right) \psi(m)-\left(\id\ot\Delta_{r} \right)\psi(m)\Big) \\
=&\tau_{13}\Big(\psi(v_{i}\trr m)\ot u_{i}-(v_{i}\trr m)\ot\Delta_{r}(u_{i})\Big)\\
=&\tau_{13}\Big((v_{j}\trr(v_{i}\trr m))\ot u_{j}\ot u_{i}-(v_{i}\trr m)\ot u_{j}\ot u_{i}v_{j}-(v_{i}\trr m)\ot v_{j}u_{i}\ot u_{j}  \Big)\\
=&\tau_{13}\Big((v_{j}\trr(v_{i}\trr m))\ot u_{j}\ot u_{i}+(u_{i}\trr m)\ot u_{j}\ot v_{i}v_{j}-(u_{i}\trr m)\ot u_{j}v_{i}\ot v_{j}  \Big)\\
=&\tau_{13}\Big( (v_{j}\trr(v_{i}\trr m))\ot u_{j}\ot u_{i}-((u_{i}u_{j})\trr m)\ot v_{i}\ot v_{j}   \Big)\\
=& u_{i}\ot u_{j}\ot  (v_{j}\trr(v_{i}\trr m))-v_{j}\ot v_{i}\ot ((u_{i}u_{j})\trr m)\\
=&u_{i}\ot u_{j}\ot  (v_{j}\trr(v_{i}\trr m))-u_{j}\ot u_{i}\ot ((v_{i}v_{j})\trr m)\\
=&u_{i}\ot u_{j}\ot  (v_{j}\trr(v_{i}\trr m))-u_{i}\ot u_{j}\ot ((v_{j}v_{i})\trr m).
\end{aligned}
$$
Thus the two sides are equal to each other if and only if
$$m \trl (v_{i}v_{j})-(m\trl v_{i})\trl v_{j}=v_{j}\trr(v_{i}\trr m)-(v_{j}v_{i})\trr m,$$
which holds obviously by $(ab) \trr v- a\trr ({b}\trr v)=(v \trl {b}) \trl a-v \trl ({b}a)$.

Then we need to prove
$$(\id-\tau_{13})\Big((\phi\ot \id)\psi(m)-(\id\ot \psi)\phi(m)\Big)=0.$$
We have
$$
\begin{aligned}
&(\phi\ot \id)\psi(m)-(\id\ot \psi)\phi(m)\\
=&(\phi\ot \id)((v_{i}\trr m)\ot u_{i})+(\id\ot \psi)(u_{i}\ot ({m\trl v_{i}}))\\
=&u_{j}\ot ((v_{i}\trr m)\trl v_{j})\ot u_{i}-u_{i}\ot (v_{j}\trr(m\trl v_{i}))\ot u_{j}\\
=&u_{i}\ot ((v_{j}\trr m)\trl v_{i})\ot u_{j}-u_{i}\ot (v_{j}\trr(m\trl v_{i}))\ot u_{j},
\end{aligned}
$$
then
$$
\begin{aligned}
&\tau_{13}\Big((\phi\ot \id)\psi(m)-(\id\ot \psi)\phi(m)\Big)\\
=&u_{i}\ot ((v_{i}\trr m)\trl v_{j})\ot u_{j}-u_{j}\ot (v_{j}\trr(m\trl v_{i}))\ot u_{i}\\
=&u_{i}\ot ((v_{i}\trr m)\trl v_{j})\ot u_{j}-u_{i}\ot (v_{i}\trr(m\trl v_{j}))\ot u_{j}.
\end{aligned}
$$
Since $M$ is an $A$-bimodule, by
$$ (a\trr m)\trl {b}-a \trr (m\trl {b}) =({b}\trr m)\trl a-{b} \trr (m\trl a),$$
we have $(\id-\tau_{13})\Big((\phi\ot \id)\psi(m)-(\id\ot \psi)\phi(m)\Big)=0$ holds. Thus $M$ is an $A$-bicomodule.

Secondly, we check that $M$ is an anti-flexible Hopf bimodule.
For (HM1), we have
$$
\begin{aligned}
&\phi(x\trr m)+\tau \psi(m\trl x)\\
=&u_{i}\ot((x\trr m)\trl v_{i})+u_{i}\ot (v_{i}\trr(m\trl x))\\
=&u_{i}\ot (x\trr(m\trl v_{i}))+u_{i}\ot ((v_{i}\trr m)\trl x)\\
=&{m}_{(-1)}\ot (x\trr {m}_{(0)})+{m}_{(1)}\ot ({m}_{(0)}\trl x).
\end{aligned}
$$
For (HM2), we have
$$
\begin{aligned}
&\psi(x\trr m)+\tau\phi(m\trl x)\\
=&(v_{i}\trr(x\trr m))\ot u_{i}+((m\trl x)\trl v_{i})\ot u_{i}\\
=&((v_{i}x)\trr m)\ot u_{i}+(m\trl (xv_{i}))\ot u_{i}\\
=&(u_{i}\trr m)\ot xv_{i}+((v_{i}x)\trr m)\ot u_{i}-(u_{i}\trr m)\ot xv_{i}\\
&-(m\trl u_{i})\ot v_{i}x+(m\trl u_{i})\ot v_{i}x+(m\trl (xv_{i}))\ot u_{i}\\
=&(u_{i}\trr m)\ot xv_{i}+((v_{i}x)\trr m)\ot u_{i}+(v_{i}\trr m)\ot xu_{i}\\
&+(m\trl v_{i})\ot u_{i}x+(m\trl u_{i})\ot v_{i}x+(m\trl (xv_{i}))\ot u_{i}\\
=&\left(x\li \trr  {m}\right) \otimes x\lii+{m}_{(0)}\otimes x{m}_{(1)}+{m}_{(0)}\otimes {m}_{(-1)}x+({m}\trl x_{2})\ot x_{1}.
\end{aligned}
$$
For (HM3), we have
$$
\begin{aligned}
 &(x\trr m\loo)\ot m_{(1)}-(m\trl x\li)\ot x\lii-m\loo\ot m_{(1)} x  \\
=&(x\trr(v_{i}\trr m))\ot u_{i}-(m\trl u_{i})\ot xv_{i}-(m\trl(v_{i}x))\ot u_{i}-(v_{i}\trr m)\ot u_{i}x\\
=&(x\trr(v_{i}\trr m))\ot u_{i}+(m\trl v_{i})\ot xu_{i}-(m\trl(v_{i}x))\ot u_{i}+(u_{i}\trr m)\ot v_{i}x\\
=&-((m\trl v_{i})\trl x)\ot u_{i}+(m\trl v_{i})\ot xu_{i}+((xv_{i})\trr m)\ot u_{i}+(u_{i}\trr m)\ot v_{i}x\\
=&\tau(-u_{i}\ot((m\trl v_{i})\trl x)+xu_{i}\ot(m\trl v_{i})+u_{i}\ot((xv_{i})\trr m)+v_{i}x\ot(u_{i}\trr m))\\
=&\tau(xm\loi\ot m\loo+x\li\ot(x\lii\trr m)-m\loi\ot (m\loo\trl x)).
\end{aligned}
$$
This completed the proof.
\end{proof}

\begin{theorem}\label{r-braided2}
 Let $(A, \cdot, r)$ be a quasitriangular  anti-flexible bialgebra. Then $A$ becomes a braided  anti-flexible bialgebra over itself with $M=A$ and $\phi: M \rightarrow A \otimes M$ and $\psi: M \rightarrow M \otimes A$ are given by
 \begin{equation}
\phi(a):=\sum_{i} u_{i} \otimes av_{i} ,\quad \psi(a):= \sum_{i} v_{i}a \otimes u_{i},
\end{equation}
\end{theorem}

\begin{proof}
All we need to do is to verify the braided compatibility conditions (BB1) and (BB2).
For (BB1), we have the right hand side is equal to the left hand side by
\begin{eqnarray*}
 &&a\li b\ot a\lii +b a\lii\ot a\li +b\li \ot ab\lii+b\lii\ot b\li a\\
&&+a\loi b\ot a\loo+b a_{(1)}\ot a\loo+b\loo\ot a b_{(1)}+b\loo\ot b\loi  a\\
&=&u_{i}b\ot av_{i}+(v_{i}a)b\ot u_{i}+b(av_{i})\ot u_{i}+bu_{i}\ot v_{i}a+u_{i}\ot a(bv_{i})+v_{i}b\ot au_{i}\\
&&+bv_{i}\ot u_{i}a+u_{i}\ot (v_{i}b)a+u_{i} b\ot a v_{i}+b u_{i}\ot v_{i} a\\
&&+v_{i} b\ot a u_{i}+b v_{i}\ot u_{i} a\\
&=&u_{i}b\ot av_{i}+(v_{i}a)b\ot u_{i}+b(av_{i})\ot u_{i}-bv_{i}\ot u_{i}a+u_{i}\ot a(bv_{i})-u_{i}b\ot av_{i}\\
&&+bv_{i}\ot u_{i}a+u_{i}\ot (v_{i}b)a+u_{i} b\ot a v_{i}+b u_{i}\ot v_{i} a\\
&&-u_{i} b\ot a v_{i}-b u_{i}\ot v_{i} a\\
&=& (v_{i}a)b\ot u_{i}+b(av_{i})\ot u_{i} +u_{i}\ot a(bv_{i})+u_{i}\ot (v_{i}b)a  \\
&=&u_{i}\ot (ab)v_{i}+v_{i}(ab)\ot u_{i}+(ba)v_{i}\ot u_{i}+u_{i}\ot v_{i}(ba)\\
&=&\Delta_{r}(ab)+\tau\Delta_{r}(ba)
\end{eqnarray*}
where we use  the conditions that $ {r}$ is  skew-symmetric and the anti-flexible Yang-Baxter equation (AFYB).
Thus (BB1) holds.

For  (BB2),  by similar computations,  we obtain
\begin{eqnarray*}
&&a\li\ot a\lii b-ba\li\ot a\lii-b\li\ot b\lii a+ab\li\ot b\lii  \\
 &&+a\loo\ot a_{(1)}  b-b a\loi\ot a\loo-b\loo\ot b_{(1)}  a+a b\loi\ot b\loo)\\
&=&u_{i}\ot (av_{i})b+v_{i}a\ot u_{i}b-bu_{i}\ot av_{i}-b(v_{i}a)\ot u_{i}-u_{i}\ot (bv_{i})a-v_{i}b\ot u_{i}a
 +au_{i}\ot bv_{i}\\
 &&+a(v_{i}b)\ot u_{i}+v_{i}  a\ot u_{i}  b - b u_{i}\ot a v_{i}-v_{i}  b\ot u_{i}  a+a u_{i}\ot b v_{i}\\
 &=&u_{i}\ot a(v_{i}b)+v_{i}a\ot u_{i}b-bu_{i}\ot av_{i}-(bv_{i})a\ot u_{i}-u_{i}\ot b(v_{i}a)-v_{i}b\ot u_{i}a+au_{i}\ot bv_{i}\\
 &&+(av_{i})b\ot u_{i}+v_{i}  a\ot u_{i}  b-b u_{i}\ot a v_{i}-v_{i}  b\ot u_{i}  a+a u_{i}\ot b v_{i}\\
 &=&(av_{i})b \ot u_{i}-v_{i}b\ot u_{i}a+au_{i}\ot bv_{i}-u_{i}\ot b(v_{i}a)-(bv_{i})a\ot u_{i}+v_{i}a\ot u_{i}b-bu_{i}\ot av_{i}\\
 &&+u_{i}\ot a(v_{i}b)-v_{i}  b\ot u_{i}  a+a u_{i}\ot b v_{i}+v_{i}  a\ot u_{i}  b-b u_{i}\ot a v_{i}\\
 &=&\tau\Big(  u_{i}\ot (av_{i})b+v_{i}a\ot u_{i}b-bu_{i}\ot av_{i}-b(v_{i}a)\ot u_{i}-u_{i}\ot (bv_{i})a-v_{i}b\ot u_{i}a+au_{i}\ot bv_{i}\\
 &&+a(v_{i}b)\ot u_{i}+v_{i}  a\ot u_{i}  b-b u_{i}\ot a v_{i}-v_{i}  b\ot u_{i}a+a u_{i}\ot b v_{i}\Big)\\
 &=&\tau\Big( a\li\ot a\lii b-ba\li\ot a\lii-b\li\ot b\lii a+ab\li\ot b\lii  \\
 &&+a\loo\ot a_{(1)}b-b a\loi\ot a\loo-b\loo\ot b_{(1)}a+a b\loi\ot b\loo\Big).
\end{eqnarray*}
Thus (BB2) holds. This completed the proof.
\end{proof}

\begin{remark}
In fact, we find that the braided terms $(a\loi\trr b)\ot a\loo+(b\trl a_{(1)})\ot a\loo+b\loo\ot (a\trl b_{(1)})+b\loo\ot (b\loi \trr a)$ and $(\id-\tau)(a\loo\ot(a_{(1)}\trr b)-(b\trl a\loi)\ot a\loo-b\loo\ot(b_{(1)}\trr a)+(a\trl b\loi)\ot b\loo))$ are equal to zero in the above proof.
Thus what we have obtained is a braided  anti-flexible bialgebra with zero braided terms.
It is an open question for us whether there exists a braided  anti-flexible bialgebra with nonzero braided terms coming from  a quasitriangular  anti-flexible bialgebra.
For this direction, one should give or classify all anti-flexible Hopf bimodule stuctures which are different from $\phi$ and $\psi$ given by us in Theorem \ref{r-braided1}.
Unfortunately, we have not found other natural anti-flexible Hopf bimodule beyond above until now.
\end{remark}

\section{Cocycle bicrossproducts of  anti-flexible bialgebras}
\subsection{Matched pair of braided   anti-flexible bialgebras}
In this section, we construct   anti-flexible bialgebra from the double cross biproduct of a matched pair of braided  anti-flexible bialgebras.

Let $A, H$ be both anti-flexible  algebras and  anti-flexible coalgebras.   For any $a, b\in A$, $x, y\in H$,  we denote maps
\begin{align*}
&\ppr: H \otimes A \to A, \quad  \ppl: A\otimes H\to A,\\
&\trr: A\otimes H \to H,  \quad   \trl: H\otimes A \to H,\\
&\phi: A \to H \otimes A, \quad  \psi: A \to A\otimes H,\\
&\rho: H  \to A\otimes H, \quad  \gamma: H \to H \otimes A,
\end{align*}
by
\begin{eqnarray*}
&& \ppr (x \otimes a) = x \ppr a, \quad \ppl(a\otimes x) = a \ppl x, \\
&& \trr (a \otimes x) = a \trr x, \quad \trl(x \otimes a) = x \triangleleft a, \\
&& \phi (a)=\sum a\lmoi\ot a\loo, \quad \psi (a) = \sum a\loo \ot a\lmi,\\
&& \rho (x)=\sum x\boi\ot x\boo, \quad \gamma (x) = \sum x\boo \ot x\bi.
\end{eqnarray*}

\begin{definition}(\cite{DBH})
A \emph{matched pair} of anti-flexible  algebras is a system $(A, \, {H},\, \trl, \, \trr, \, \ppl, \, \ppr)$ consisting
of two anti-flexible  algebras $A$ and  ${H}$ and four bilinear maps $\triangleleft : {H}\otimes A\to {H}$, $\trr: {A} \otimes H
\to H$, $\ppl: A \otimes {H} \to A$, $\ppr: H\otimes {A} \to {A}$ such that $({H}, \, \trr, \, \trl)$ is an $A$-bimodule,  $(A, \, \ppr, \, \ppl)$ is an ${H}$-bimodule and satisfying the following compatibilities:
\begin{enumerate}
\item[(AM1)] $x \ppr(a b)+(ba)\ppl x =(x\ppr a) b+b(a\ppl x)+(x\trl a) \ppr b+b\ppl(a\trr x)$,
\item[(AM2)] $(a\ppl x)b+(a\trr x)\ppr b -a(x\ppr b)-a\ppl(x\trl b)$\\
$=(b\trr x)\ppr a+(b\ppl x)a-b(x\ppr a)-b\ppl (x\trl a)$,
\item[(AM3)]$(x y) \trl a+a \trr (yx)=x(y \trl a)+(a \trr y)x+x\trl(y\ppr a)+(a\ppl y)\trr x$,
\item[(AM4)]$ (x\trl a) y+(x\ppr a) \trr y-x(a\trr y)-x\trl(a\ppl y)$\\
$=(y\trl a)x+(y\ppr a)\trr x-y(a\trr x)-y\trl(a\ppl x)$.
\end{enumerate}
\end{definition}

\begin{lemma}(\cite{DBH})
Let  $(A, \, {H}, \, \trl, \, \trr, \, \ppl, \, \ppr)$ be a matched pair of  anti-flexible   algebras.
Then $A \, \bowtie {H}:= A \oplus  {H}$, as a vector space, with the multiplication defined for all $a, b\in A$ and $x, y\in {H}$ by
\begin{equation}
(a, x) (b, y): = \big(ab+ a \ppl y + x\ppr b,\, \, a\trr y + x\trl b + xy \big)
\end{equation}
is an  anti-flexible  algebra called the \emph{bicrossed product} associated to the matched pair of   anti-flexible  algebras $A$ and ${H}$.
\end{lemma}

Now we introduce the notion of matched pairs of  anti-flexible  coalgebras, which is the dual version of  matched pairs of  anti-flexible algebras.

\begin{definition} A \emph{matched pair} of  anti-flexible   coalgebras is a system $(A, \, {H}, \, \phi, \, \psi, \, \rho, \, \gamma)$ consisting of two anti-flexible   coalgebras $A$ and ${H}$ and four bilinear maps
$\phi: {A}\to H\otimes A$, $\psi: {A}\to A \otimes H$, $\rho: H\to A\otimes {H}$, $\gamma: H \to {H} \ot {A}$
such that $({H}, \, \rho, \, \gamma)$ is an $A$-bicomodule,  $(A, \, \phi, \, \psi)$ is an ${H}$-bicomodule and satisfying the following compatibility conditions:

\begin{enumerate}
\item[(MC1)] $\phi\left(a_{1}\right) \otimes a_{2}+\gamma\left(a_{(-1)}\right) \otimes a_{(0)}-a_{(-1)} \otimes \Delta_{A}\left(a_{(0)}\right)\\
    =\tau_{13}\left(\Delta_{A}\left(a_{(0)}\right) \otimes a_{(1)} -a_{1} \otimes \psi\left(a_{2}\right)-a_{(0)} \otimes \rho\left(a_{(1)}\right)\right)$,

\item[(MC2)] $\psi\left(a_{1}\right) \otimes a_{2}+\rho\left(a_{(-1)}\right) \otimes a_{(0)}-a_{1} \otimes \phi(a_{2})-a_{(0)} \otimes \gamma(a_{(1)})\\
     =\tau_{13}\left(\psi\left(a_{1}\right) \otimes a_{2}+\rho\left(a_{(-1)}\right) \otimes a_{(0)}-a_{1} \otimes \phi(a_{2})-a_{(0)} \otimes \gamma(a_{(1)})\right)$,
\item[(MC3)] $\rho\left(x_{1}\right) \otimes x_{2}+\psi\left(x_{[-1]}\right) \otimes x_{[0]}-x_{[-1]} \otimes \Delta_{H}\left( x_{[0]}\right)  \\
    =\tau_{13}\left(\Delta_{H}(x\boo)\ot x\bi-x\boo\ot \phi(x\bi)-x_1\ot \gamma(x_2) \right)$,

\item[(MC4)] $\gamma\left(x_{1}\right) \otimes x_{2}+\phi\left(x_{[-1]}\right) \otimes x_{[0]}-x_{[0]} \otimes \psi\left(x_{[1]}\right)-x_{1} \otimes \rho(x_{2})\\
=\tau_{13}\left(\gamma\left(x_{1}\right) \otimes x_{2}+\phi\left(x_{[-1]}\right) \otimes x_{[0]}-x_{[0]} \otimes \psi\left(x_{[1]}\right)-x_{1} \otimes \rho(x_{2})\right)$.
\end{enumerate}
\end{definition}

\begin{lemma}\label{lem1} Let $(A, \, {H}, \, \phi, \, \psi, \, \rho, \, \gamma)$ be a matched pair of  anti-flexible   coalgebras. We define $E=A\lrcoprod H$ as the vector space $A\oplus H$ with   comultiplication
$$\Delta_{E}(a)=(\Delta_{A}+\phi+\psi)(a),\quad\Delta_{E}(x)=(\Delta_{H}+\rho+\gamma)(x),$$
that is
$$\Delta_{E}(a)=\sum a\li \ot a\lii+\sum a\loi \ot a\loo+\sum a\mo\ot a\mi, $$
$$\Delta_{E}(x)=\sum x\li \ot x\lii+\sum  x\boi \ot x\boo+\sum x\boo \ot x\bi.$$
Then  $A\lrcoprod H$ is an   anti-flexible  coalgebra which is called the \emph{bicrossed coproduct} associated to the matched pair of anti-flexible  coalgebras $A$ and $H$.
\end{lemma}

\begin{proof} We need to prove that
 $$(\Delta_E \otimes \id) \Delta_E(a, x)-(\id \otimes \Delta_E) \Delta_E(a, x)=\tau_{13}\left((\Delta_E \otimes \id) \Delta_E(a, x)-(\id \otimes \Delta_E) \Delta_E(a, x) \right).$$ 
 The left hand side is equal to
\begin{eqnarray*}
& &(\Delta_E \otimes \id) \Delta_E(a, x)-(\id \otimes \Delta_E) \Delta_E(a, x)\\
&=&(\Delta_E\otimes \id)\left(a_{1} \otimes a_{2}+a_{(-1)} \otimes a_{(0)}+a_{(0)} \otimes a_{(1)}+x_{1} \otimes x_{2}+x_{[-1]} \otimes x_{[0]}+x_{[0]} \otimes x_{[1]}\right)\\
&&-(\id \otimes \Delta_E)\left(a_{1} \otimes a_{2}+a_{(-1)} \otimes a_{(0)}+a_{(0)} \otimes a_{(1)}+x_{1} \otimes x_{2}+x_{[-1]} \otimes x_{[0]}+x_{[0]} \otimes x_{[1]}\right)\\
&=&\Delta_{A}\left(a_{1}\right) \otimes a_{2}+\phi\left(a_{1}\right) \otimes a_{2}+\psi\left(a_{1}\right) \otimes a_{2}\\
&&+\Delta_{H}\left(a_{(-1)}\right) \otimes a_{(0)}+\rho\left(a_{(-1)}\right) \otimes a_{(0)}+\gamma\left(a_{(-1)}\right) \otimes a_{(0)}\\
&&+\Delta_{A}\left(a_{(0)}\right) \otimes a_{(1)}+\phi\left(a_{(0)}\right) \otimes a_{(1)}+\psi\left(a_{(0)}\right) \otimes a_{(1)}\\
&&+\Delta_{H}\left(x_{1}\right) \otimes x_{2}+\rho\left(x_{1}\right) \otimes x_{2}+\gamma\left(x_{1}\right) \otimes x_{2}\\
&&+\Delta_{A}\left(x_{[-1])}\right) \otimes x_{[0]}+\phi\left(x_{[-1]}\right) \otimes x_{[0]}+\psi\left(x_{[-1]}\right) \otimes x_{[0]}\\
&&+\Delta_{H}\left(x_{[0]}\right) \otimes x_{[1]}+\rho\left(x_{[0]}\right) \otimes x_{[1]}+\gamma\left(x_{[0]}\right) \otimes x_{[1]}\\
&&-a_{1} \otimes \Delta_{A}\left(a_{2}\right)-a_{1} \otimes \phi\left(a_{2}\right)-a_{1} \otimes \psi\left(a_{2}\right)\\
&&-a_{(-1)} \otimes \Delta_{A}\left(a_{(0)}\right)-a_{(-1)} \otimes \phi\left(a_{(0)}\right)-a_{(-1)} \otimes \psi\left(a_{(0)}\right)\\
&&-a_{(0)} \otimes \Delta_{H}\left(a_{(1)}\right)-a_{(0)} \otimes \rho\left(a_{(1)}\right)-a_{(0)} \otimes \gamma\left(a_{(1)}\right)\\
&&-x_{1} \otimes \Delta_{H}\left(x_{2}\right)-x_{1} \otimes \rho\left(x_{2}\right)-x_{1} \otimes \gamma\left(x_{2}\right)\\
&&-x_{[-1]} \otimes \Delta_{H}\left(x_{[0]}\right)-x_{[-1]} \otimes \rho\left(x_{[0]}\right)-x_{[-1]} \otimes \gamma\left(x_{[0]}\right)\\
&&-x_{[0]} \otimes \Delta_{A}\left(x_{[1]}\right)-x_{[0]} \otimes \phi\left(x_{[1]}\right)-x_{[0]} \otimes \psi\left(x_{[1]}\right).
\end{eqnarray*}
The right hand side can be computed similarly. Thus the two side are equal to each other  if and only if $(A, \, {H}, \, \phi, \, \psi, \, \rho, \, \gamma)$ is a matched pair of  anti-flexible coalgebras.
\end{proof}

In the following of this section, we construct  anti-flexible bialgebra from the double cross biproduct of a matched pair of braided anti-flexible bialgebras.
First we generalize the concept of anti-flexible Hopf bimodule to the case of $A$ is not necessarily an anti-flexible bialgebra but it is both an anti-flexible algebra and an anti-flexible  coalgebra.
But by abuse of notation, we also call it anti-flexible Hopf bimodule.
\begin{definition}
 Let $A$ be an anti-flexible algebra and an anti-flexible  coalgebra . An anti-flexible Hopf bimodule over $A$ is a space $V$ endowed with maps
\begin{align*}
&\trr: A\otimes V \to V, \quad \trl: V\otimes A \to V,\\
&\phi:V \to A \otimes V, \quad  \psi: V \to V\otimes A,
\end{align*}
such that $V$ is simultaneously a bimodule, a  bicomodule over $A$ and satisfying
 the following compatibility conditions
 \begin{enumerate}
\item[(HM4)]$\phi(a \trr v)+\tau\psi(v \trl a)=v_{(-1)}\ot (a\trr v_{(0)})+v_{(1)}\ot (v_{(0)}\trl a)$,
\item[(HM5)]$\psi(a \trr v)+\tau\phi(v\trl a)=\left(a\li \trr  v\right) \otimes a\lii+v_{(0)}\otimes av_{(1)}+v_{(0)}\otimes v_{(-1)}a+(v\trl a_{2})\ot a_{1}$,

\item[(HM6)] $ (a\trr v\loo)\ot v_{(1)}-(v\trl a\li)\ot a\lii-v\loo\ot v_{(1)} a  $\\
$=\tau(  av\loi\ot v\loo+a\li\ot(a\lii\trr v)-v\loi\ot (v\loo\trl a))$.
\end{enumerate}
then $V$ is called an anti-flexible Hopf bimodule over ${A}$.
\end{definition}
We denote  the  category of  anti-flexible Hopf bimodules over ${A}$ by ${}^{A}_{A}\mathcal{M}{}^{A}_{A}$.

\begin{definition}
If $A$ be an anti-flexible algebra and anti-flexible  coalgebra and  $H$ is an anti-flexible Hopf bimodule over $A$. If $H$ is an anti-flexible algebra and an anti-flexible coalgebra in ${}^{A}_{A}\mathcal{M}^{A}_{A}$, then we call $H$ a \emph{braided anti-flexible bialgebra} over $A$, if the following conditions are satisfied:

\begin{enumerate}
\item[(BB3)]
$\Delta_{H}(xy)+\tau\Delta_{H}(yx)$\\
$=x\li y\ot x\lii +y x\lii\ot x\li +y\li \ot xy\lii+y\lii\ot y\li x$\\
$+(x\loi\trr y)\ot x\loo+(y\trl x_{(1)})\ot x\loo+y\loo\ot (x\trl y_{(1)})+y\loo\ot (y\loi \trr x)$,
\end{enumerate}
\begin{enumerate}
\item[(BB4)]
$(\id-\tau)\Big(x\li\ot x\lii y-yx\li\ot x\lii-y\li\ot y\lii x+xy\li\ot y\lii  \Big)$\\
 $+(\id-\tau)\Big(x\loo\ot(x_{(1)}\trr y)-(y\trl x\loi)\ot x\loo-y\loo\ot(y_{(1)}\trr x)+(x\trl y\loi)\ot y\loo  \Big)=0$.
\end{enumerate}

\end{definition}

\begin{definition}\label{def:dmp}
Let $A, H$ be both  anti-flexible algebras and  anti-flexible coalgebras. If  the following conditions hold:
\begin{enumerate}
\item[(DM1)]  $\phi(ab)+\tau\psi(ba)$\\
$= (b \trr a_{(1)}) \otimes a_{(0)} +b_{(1)}\ot b_{(0)} a+(a_{(-1)} \trl b) \otimes a_{(0)}+b_{(-1)} \otimes a b_{(0)} ,$

\item[(DM2)] $\rho(x y)+\tau \gamma( y x) $\\
$= \left(y \ppr x_{[1]}\right) \otimes x_{[0]}+y_{[1]} \otimes y_{[0]} x +(x_{[-1]}\ppl y)\ot x_{[0]}+y_{[-1]}\ot xy_{[0]} , $

\item[(DM3)] $\Delta_{A}(x \ppr b)+\tau\Delta_{A}(b \ppl x)$ \\
$=\left(b\ppl x_{[0]}\right) \otimes x_{[-1]} +b_{2} \otimes\left(b_{1} \ppl x\right) +\left(x_{[0]} \ppr b\right)\otimes x_{[1]} +b\li \ot (x\ppr b\lii) ,$

\item[(DM4)] $\Delta_{H}(a \trr x)+\tau \Delta_{H}(x \trl a)$\\
$= (x\trl a\loo)\ot a_{(-1)} +x_{2} \otimes\left(x_{1}  \trl a\right)+\left(a_{(0)}\trr x\right)\otimes a_{(1)} +x\li\ot (a\trr x\lii)$

\item[(DM5)]
 $\phi(x \ppr b)+\tau\psi(b \ppl x)+\gamma(x\trl b)+\tau \rho(b\trr x)$\\
 $=(b\trr x\poo)\ot x\boi+b_{(1)}\ot (b\loo\ppl x)+(x_{[0]}\trl b)\otimes x_{[1]}+b_{(-1)} \otimes (x \ppr b_{(0)}), $

\item[(DM6)]
$\psi(x \ppr b)+\tau\phi(b \ppl x)+\rho(x\trl b)+\tau\gamma(b\trr x)$\\
$=(b\ppl x\lii) \ot x\li+ bx\bi\ot x\poo+b\lii\ot (b\li \trr x)+b\loo\ot b\loi x$\\
$+(x\li \ppr b) \otimes x\lii+x_{[-1]} b \otimes x_{[0]}+b\li\ot (x\trl b\lii)+b\loo\ot xb_{(1)}, $
\end{enumerate}
\begin{enumerate}


\item[(DM7)]$a\loo\ot(a_{(1)}\trl b)+ab\loo\ot b_{(1)}-b\loo\ot(b_{(1)}\trl a)-ba\loo\ot a_{(1)}$\\
   $=\tau(a\loi\ot a\loo b+(a\trr b\loi)\ot b\loo-b\loi\ot b\loo a-(b\trr a\loi)\ot a\loo)$,




\item[(DM8)]$(\id-\tau)(x\li\ot(x\lii\trl b) +(x\trl b\loo)\ot b_{(1)}-b\loi\ot(b\loo\trr x)-(b\trr x\li)\ot x\lii )=0$,

\item[(DM9)]$x\boo\ot(x\bi\ppl y)+xy\boo\ot y\bi-y\boo\ot(y\bi\ppl x)-yx\boo\ot x\bi$\\
$=\tau(x\boi\ot x\boo y +(x\ppr y\boi)\ot y\boo-y\boi\ot y\boo x-(y\ppr x\boi)\ot x\boo)$,

\item[(DM10)]$(\id-\tau)(x\boi\ot(x\boo\ppr b) +(x\ppr b\li)\ot b\lii-b\li\ot (b\lii \ppl x)-(b\ppl x\boo)\ot x\bi )=0$,

\item[(DM11)]
$b\loi\ot (b\loo\ppl x)+(b\trr x\boo)\ot x\bi)
-x\li\ot (x\lii \ppr b)-x\boo\ot x\bi b-(x\trl b\li)\ot b\lii-xb\loi\ot b\loo$\\
$=\tau(
  b\li\ot(b\lii\trr x)+b\loo\ot b_{(1)}x+(b\ppl x\li)\ot x\lii+bx\boi\ot x\boo-x\boi\ot(x\boo\trl b) -(x\ppr b\loo)\ot b_{(1)})$,
\end{enumerate}
\noindent then $(A, H)$ is called a \emph{double matched pair}.
\end{definition}

\begin{theorem}\label{main1}Let $(A, H)$ be matched pair of  anti-flexible algebras and anti-flexible  coalgebras,
$A$ is  a  braided   anti-flexible bialgebra in
${}^{H}_{H}\mathcal{M}^{H}_{H}$, $H$ is  a braided    anti-flexible bialgebra in
${}^{A}_{A}\mathcal{M}^{A}_{A}$. If we define the double cross biproduct of $A$ and
$H$, denoted by $A\lrbiprod H$, $A\lrbiprod H=A\bowtie H$ as
 anti-flexible algebra, $A\lrbiprod H=A\lrcoprod H$ as anti-flexible  coalgebra, then
$A\lrbiprod H$ becomes an   anti-flexible bialgebra if and only if  $(A, H)$ forms a double matched pair.
\end{theorem}

Theorem \ref{main1} is a special case of Theorem \ref{main2} in next section, and the proof is by direct computations, so we omit the details.

\subsection{Cocycle bicrossproduct   anti-flexible bialgebras}

In this section, we construct cocycle bicrossproduct   anti-flexible bialgebras, which is a generalization of double cross biproduct.

Let $A, H$ be both  anti-flexible algebras and  anti-flexible coalgebras.   For $a, b\in A$, $x, y\in H$,  we denote maps
\begin{align*}
&\sigma: H\otimes H \to A,\quad \theta: A\otimes A \to H,\\
&P: A  \to H\otimes H,\quad  Q: H \to A\otimes A,
\end{align*}
by
\begin{eqnarray*}
&& \qquad\qquad\sigma (x,y)  \in  H, \quad \theta(a, b) \in A,\\
&& P(a)=\sum a\ppi\ot  a\pii, \quad Q(x) = \sum x\qi \ot x\qii.
\end{eqnarray*}

A bilinear map $\si: H\ot H\to A$ is called a cocycle on $H$ if
\begin{enumerate}
\item[(CC1)] $\sigma(x y, z)-{\sigma}(x, y z)+\sigma(x, y)\ppl z-x \ppr \sigma(y, z)\\
=\sigma(z y, x)-{\sigma}(z, y x)+\sigma(z, y)\ppl x-z \ppr \sigma(y, x).$
\end{enumerate}

A bilinear map $\theta: A\ot A\to H$ is called a cocycle on $A$ if
\begin{enumerate}
\item[(CC2)] $\theta(a b, c)-\theta(a, b c)+\theta(a, b) \triangleleft c-a\trr \theta(b, c)=\theta(c b, a)-\theta(c, b a)+\theta(c, b) \trl a-c\trr \theta(b, a).$
\end{enumerate}

A bilinear map $P: A\to H\ot H$ is called a cycle on $A$ if
\begin{enumerate}
\item[(CC3)]  $\Delta_H(a\ppi)\ot a\pii-a\ppi\ot \Delta_H(a\pii)+P(a\lmoo)\ot a\mi-a\lmoi\ot P(a\lmoo)\\
=\tau_{13}\left(\Delta_H(a\ppi)\ot a\pii-a\ppi\ot \Delta_H(a\pii)+P(a\lmoo)\ot a\mi-a\lmoi\ot P(a\lmoo)\right)$.
\end{enumerate}

A bilinear map $Q: H\to A\ot A$ is called a cycle on $H$ if
\begin{enumerate}
\item[(CC4)]  $\Delta_A(x\qi)\ot x\qii-x\qi\ot \Delta_A(x\qii)+Q(x\boo)\ot x\bi-x\boi\ot Q(x\boo)\\
=\tau_{13}\left( \Delta_A(x\qi)\ot x\qii-x\qi\ot \Delta_A(x\qii)+Q(x\boo)\ot x\bi-x\boi\ot Q(x\boo)\right)$.
\end{enumerate}

In the following definitions, we introduced the concept of cocycle
 anti-flexible algebras and cycle  anti-flexible coalgebras, which are  in fact not really
ordinary  anti-flexible algebras and  anti-flexible coalgebras, but generalized ones.

\begin{definition}
(i): Let $\si$ be a cocycle on a vector space  $H$ equipped with multiplication $H \ot H \to H$, satisfying the following cocycle identity:
\begin{enumerate}
\item[(CC5)] $(x y) z-x(y z)+\sigma(x, y) \trr z-x \triangleleft \sigma(y, z)=(z y) x-z(y x)+\sigma(z, y) \trr x-z \trl \sigma(y, x)$.
\end{enumerate}
Then  $H$ is called a   $\si$-anti-flexible algebra which is denoted by $(H, \, \sigma)$.

(ii): Let $\theta$ be a cocycle on a vector space $A$  equipped with a multiplication $A \ot A \to A$, satisfying the the
following cocycle  identity:
\begin{enumerate}
\item[(CC6)] $(a b) c-a(b c)+\theta(a, b) \ppr c-a\ppl \theta(b, c)=(c b) a-c(b a)+\theta(c, b) \ppr a-c\ppl \theta(b, a)$.
\end{enumerate}
Then  $A$ is called a $\theta$-anti-flexible algebra which is denoted by $(A, \, \theta)$.

(iii) Let $P$ be a cycle on a vector space  $H$ equipped with a comultiplication $\Delta: H \to H \ot H$, satisfying the the
following cycle  identity:
\begin{enumerate}
\item[(CC7)] $\Delta_H(x\li)\ot x\lii-x_1\ot \Delta_H(x_2)+ P(x\boi) \ot x\boo-x\boo\ot P(x\bi)\\
=\tau_{13} \left(\Delta_H(x\li)\ot x\lii-x_1\ot \Delta_H(x_2)+ P(x\boi) \ot x\boo-x\boo\ot P(x\bi) \right)$.
\end{enumerate}
\noindent Then  $H$ is called a   $P$-anti-flexible coalgebra which is denoted by $(H, \,  P)$.

(iv) Let $Q$ be a cycle on a vector space  $A$ equipped with a comultiplication $\Delta: A \to A \ot A$, satisfying the the
following cycle identity:
\begin{enumerate}
\item[(CC8)] $\Delta_{A}(a\li)\ot a\lii-a\li\ot \Delta_A(a\lii)+Q(a\lmoi)\ot a\lmoo-a\mo\ot Q(a\mi)\\
=\tau_{13} \left( \Delta_{A}(a\li)\ot a\lii-a\li\ot \Delta_A(a\lii)+Q(a\lmoi)\ot a\lmoo-a\mo\ot Q(a\mi) \right)$.
\end{enumerate}
\noindent Then  $A$ is called a  $Q$-anti-flexible coalgebra which is denoted by $(A, \,  Q)$.
\end{definition}

\begin{definition}
A  \emph{cocycle cross product system } is a pair of  $\theta$-anti-flexible algebra $A$ and $\sigma$-anti-flexible algebra $H$,
where $\si: H\ot H\to A$ is a cocycle on $H$,  $\theta: A\ot A\to H$ is a cocycle on $A$ and the following conditions are satisfied:
\begin{enumerate}
\item[(CP1)] $x \ppr(a b)+(ba)\ppl x +\sigma( x,\theta(a, b))+\sigma(\theta(b, a), x)$\\
$=(x\ppr a) b+b(a\ppl x)+(x\trl a) \ppr b+b\ppl(a\trr x)$,

\item[(CP2)] $(a\ppl x)b+(a\trr x)\ppr b -a(x\ppr b)-a\ppl(x\trl b)$\\
$=(b\trr x)\ppr a+(b\ppl x)a-b(x\ppr a)-b\ppl (x\trl a)$,

\item[(CP3)] $(x y) \ppr a+\sigma(x, y) a -x\ppr(y\ppr a)-\sigma(x, y\trl a)$\\
$=(a\ppl y)\ppl x-a\ppl (yx)+\sigma(a\trr y, x)-a\sigma(y, x)$,

\item[(CP4)] $(x\ppr a)\ppl y+\sigma(x\trl a, y)-x\ppr (a\ppl y)-\sigma(x, a\trr y)$\\
$=(y\ppr a)\ppl x+\sigma(y\trl a, x)-y\ppr(a\ppl x)-\sigma(y, a\trr x)$,

\item[(CP5)]$(x y) \trl a+a \trr (yx)+\theta(\sigma(x, y), a)+\theta(a, \sigma(y, x))$\\
$=x(y \trl a)+(a \trr y)x+x\trl(y\ppr a)+(a\ppl y)\trr x$,

\item[(CP6)]$ (x\trl a) y+(x\ppr a) \trr y-x(a\trr y)-x\trl(a\ppl y)$\\
$=(y\trl a)x+(y\ppr a)\trr x-y(a\trr x)-y\trl(a\ppl x)$,

\item[(CP7)] $(ab) \trr x+\theta(a, b) x-a\trr (b\trr x)-\theta(a, b\ppl x)$\\
$=(x\trl b)\trl a+\theta(x\ppr b, a)-x\theta(b, a)-x\trl(ba)$,

\item[(CP8)] $(a\trr x)\trl b+\theta(a\ppl x, b)-a\trr(x\trl b)-\theta(a, x\ppr b)$\\
$=(b\trr x)\trl a+\theta(b\ppl x, a)-b\trr(x\trl a)-\theta(b, x\ppr a)$.
\end{enumerate}
\end{definition}

\begin{lemma}
Let $(A, H)$ be  a  cocycle cross product system.
If we define $E=A_{\sigma}\#_{\theta} H$ as the vector space $A\oplus H$ with the   multiplication
\begin{align}
(a, x)(b, y)=\big(ab+x\ppr b+a\ppl y+\sigma(x, y), \, xy+x\trl b+a\trr y+\theta(a, b)\big),
\end{align}
then $E=A_{\sigma}\#_{\theta} H$ forms an  anti-flexible algebra  which is called the cocycle cross product anti-flexible algebra.
\end{lemma}

\begin{proof} We have to check that
$$((a, x)(b, y))(c, z)-(a, x)((b, y)(c, z))=((c, z)(b, y))(a, x)-(c, z)((b, y)(a, x)). $$
By direct computations, the left hand side is equal to
\begin{eqnarray*}
&&\big((a, x)(b, y)\big)(c, z)-(a, x)\big((b, y)(c, z)\big)\\
&=&\big(a b+x \ppr b+a\ppl y+\sigma(x, y), x y+x \trl b+a\trr y+\theta(a, b)\big)(c, z)\\
&&-\left(a, x\right)\big(b c+y\ppr c+b\ppl z+\sigma(y, z), y z+y \trl c+b \trr z+\theta(b, c)\big)\\
&=&\Big( (a b) c+(x\ppr b) c+(a\ppl y) c+\sigma(x, y) c+(x y) \ppr c +(x \trl b) \ppr c+(a \trr y)\ppr c\\
&&+\theta(a, b) \ppr c+(a b) \ppl z+(x \ppr b)\ppl z+(a\ppl y)\ppl z+\sigma(x, y)\ppl z\\
&&+\sigma(x y, z)+\sigma(x \trl b, z)+\sigma(a \trr y, z) +\sigma(\theta(a, b), z),  \quad (x y) z+(x\trl b) z+(a\trr y) z\\
&&+\theta(a, b) z+(x y) \trl c+(x \trl b) \trl c+(a \trr y) \trl c+\theta(a, b) \trl c+(a b) \trr z+(x \ppr b) \trr z \\
&&+(a\ppl y)\trr z+\sigma(x, y) \trr z+\theta(a b, c)+\theta(x \ppr b, c)+\theta(a\ppl y, c)+\theta(\sigma(x, y), c)\Big)\\
&&-\Big(a(b c)+a(y \ppr c)+a(b\ppl z)+a(\sigma(y, z))+x \ppr(b c)+x \ppr(y\ppr c)\\
&&+x \ppr(b\ppl z)+x \ppr \sigma(y, z)+a\ppl (y z)+a\ppl (y \trl c)+a\ppl (b \trr z)+a\ppl \theta(b, c)\\
&&+\sigma(x, y z)+\sigma(x, y \trl c)+\sigma(x, b \trr z)+\sigma(x, \theta(b, c)), \quad x(y z)+x(y \trl c)+x(b \trr z)\\
&&+x \theta(b, c)+x \trl (b c)+x \trl(y \ppr c)+x\trl (b\ppl z)+x \trl \sigma(y, z)+a \trr(y z)+a \trr(y\trl c)\\
&&+a \trr(b \trr z)+a \trr \theta(b, c)+\theta(a, b c)+\theta(a, y \ppr c)+\theta(a, b\ppl z)+\theta(a, \sigma(y, z))\Big).
\end{eqnarray*}
We can compute the right hand side in the same way. Thus the two sides are equal to each other if and only if (CP1)--(CP8) hold.
\end{proof}

\begin{definition}
A  \emph{cycle cross coproduct system } is a pair of   $P$-anti-flexible coalgebra $A$ and  $Q$-anti-flexible coalgebra $H$,  where $P: A\to H\ot H$ is a cycle on $A$,  $Q: H\to A\ot A$ is a cycle on $H$ such that following conditions are satisfied:
\begin{enumerate}
\item[(CCP1)]$\phi\left(a_{1}\right) \otimes a_{2}+\gamma\left(a_{(-1)}\right) \otimes a_{(0)}-a_{(-1)} \otimes \Delta_{A}\left(a_{(0)}\right)-a\ppi \otimes Q\left(a\pii\right)\\
    =\tau_{13}\left(\Delta_{A}\left(a_{(0)}\right) \otimes a_{(1)}+Q(a\ppi)\ot a\pii -a_{1} \otimes \psi\left(a_{2}\right)-a_{(0)} \otimes \rho\left(a_{(1)}\right)\right)$,

\item[(CCP2)] $\psi\left(a_{1}\right) \otimes a_{2}+\rho\left(a_{(-1)}\right) \otimes a_{(0)}-a_{1} \otimes \phi(a_{2})-a_{(0)} \otimes \gamma(a_{(1)})\\
     =\tau_{13}\left(\psi\left(a_{1}\right) \otimes a_{2}+\rho\left(a_{(-1)}\right) \otimes a_{(0)}-a_{1} \otimes \phi(a_{2})-a_{(0)} \otimes \gamma(a_{(1)})\right)$,
\item[(CCP3)] $ \Delta_{H}\left(a_{(-1)}\right) \otimes a_{(0)}+P\left(a_{1}\right) \otimes a_{2}-a_{(-1)} \otimes \phi\left(a_{(0)}\right)-a\ppi\otimes \gamma\left(a\pii\right) $\\
$=\tau_{13}(\psi\left(a_{(0)}\right) \otimes a_{(1)}+\rho\left(a\ppi\right) \otimes a\pii -a_{(0)}\ot \Delta_{H}\left(a_{(1)}\right)-a_{1}\otimes P\left(a_{2}\right) )$,
\item[(CCP4)] $ a_{(-1)} \otimes \psi\left(a_{(0)}\right)+a\ppi \otimes \rho\left(a\pii\right)-\phi\left(a_{(0)}\right) \otimes a_{(1)}-\gamma\left(a\ppi\right) \otimes a\pii $\\
    $=\tau_{13}\left(  a_{(-1)} \otimes \psi\left(a_{(0)}\right)+a\ppi \otimes \rho\left(a\pii\right)-\phi\left(a_{(0)}\right) \otimes a_{(1)}-\gamma\left(a\ppi\right) \otimes a\pii \right)$,

\item[(CCP5)] $\rho\left(x_{1}\right) \otimes x_{2}+\psi\left(x_{[-1]}\right) \otimes x_{[0]}-x_{[-1]} \otimes \Delta_{H}\left( x_{[0]}\right) -x\qi\ot P(x\qii) \\
    =\tau_{13}\left(\Delta_{H}(x\boo)\ot x\bi+P(x\qi)\ot x\qii-x\boo\ot \phi(x\bi)-x_1\ot \gamma(x_2) \right)$,

\item[(CCP6)] $\gamma\left(x_{1}\right) \otimes x_{2}+\phi\left(x_{[-1]}\right) \otimes x_{[0]}-x_{[0]} \otimes \psi\left(x_{[1]}\right)-x_{1} \otimes \rho(x_{2})\\
=\tau_{13}\left(\gamma\left(x_{1}\right) \otimes x_{2}+\phi\left(x_{[-1]}\right) \otimes x_{[0]}-x_{[0]} \otimes \psi\left(x_{[1]}\right)-x_{1} \otimes \rho(x_{2})\right)$,

\item[(CCP7)] $\Delta_{A}\left(x_{[-1]}\right) \otimes x_{[0]} +Q\left(x_{1}\right) \otimes x_{2}-x_{[-1]} \otimes \rho\left(x_{[0]}\right)-x\qi \otimes \psi\left(x\qii\right) \\
    =\tau_{13}\left( \gamma(x\boo)\ot x\bi+\phi(x\qi)\ot x\qii -x\boo\ot\Delta_A(x\bi)-x_1\ot Q(x_2) \right)$,

\item[(CCP8)] $\rho(x\boo)\ot x\bi+\psi(x\qi)\ot x\qii-x\boi\ot\gamma(x\boo)-x\qi\ot\phi(x\qii)\\
=\tau_{13}\left(\rho(x\boo)\ot x\bi+\psi(x\qi)\ot x\qii-x\boi\ot\gamma(x\boo)-x\qi\ot\phi(x\qii)\right)$.
\end{enumerate}
\end{definition}

\begin{lemma}\label{lem2} Let $(A, H)$ be  a  cycle cross coproduct system. If we define $E=A^{P}\# {}^{Q} H$ as the vector space $A\oplus H$ with the   comultiplication
$$\Delta_{E}(a)=(\Delta_{A}+\phi+\psi+P)(a),\quad \Delta_{E}(x)=(\Delta_{H}+\rho+\gamma+Q)(x), $$
that is
$$\Delta_{E}(a)= a\li \ot a\lii+ a\moi \ot a\mo+a\mo\ot a\mi+a\ppi\ot a\pii,$$
$$\Delta_{E}(x)= x\li \ot x\lii+ x\boi \ot x\boo+x\boo \ot x\bi+x\qi\ot x\qii,$$
then  $A^{P}\# {}^{Q} H$ forms an   anti-flexible coalgebra which we will call it the cycle cross coproduct anti-flexible  coalgebra.
\end{lemma}

\begin{proof} We have to check
$$ (\Delta_{E} \otimes \id) \Delta_{E}(a, x)-(\id\otimes \Delta_{E} \otimes) \Delta_{E}(a, x)=\tau_{13}\left( (\Delta_{E} \otimes \id) \Delta_{E}(a, x)-(\id\otimes \Delta_{E} \otimes) \Delta_{E}(a, x) \right) .$$
By direct computations, the left hand side is equal to
\begin{eqnarray*}
&& (\Delta_{E} \otimes \id) \Delta_{E}(a, x)-(\id\otimes \Delta_{E}  ) \Delta_{E}(a, x)\\
&=&\Delta_{A}\left(a_{1}\right) \otimes a_{2}+\phi\left(a_{1}\right) \otimes a_{2}+\psi\left(a_{1}\right) \otimes a_{2}+P\left(a_{1}\right) \otimes a_{2}\\
&&+\Delta_{H}\left(a_{(-1)}\right) \otimes a_{(0)}+\rho\left(a_{(-1)}\right) \otimes a_{(0)}+\gamma\left(a_{(-1)}\right) \otimes a_{(0)}+Q\left(a_{(-1)}\right) \otimes a_{(0)}\\
&&+\Delta_{A}\left(a_{(0)}\right) \otimes a_{(1)}+\phi\left(a_{(0)}\right) \otimes a_{(1)}+\psi\left(a_{(0)}\right) \otimes a_{(1)}+P\left(a_{(0)}\right) \otimes a_{(1)}\\
&&+\Delta_{H}\left(a\ppi\right) \otimes a\pii+\rho\left(a\ppi\right) \otimes a\pii+\gamma\left(a\ppi\right) \otimes a\pii+Q\left(a\ppi\right) \otimes a\pii\\
&&+\Delta_{H}\left(x_{1}\right) \ot x_{2}+\rho\left(x_{1}\right) \otimes x_{2}+\gamma\left(x_{1}\right) \otimes x_{2}+Q\left(x_{1}\right) \otimes x_{2}\\
&&+\Delta_{A}\left(x\boi\right) \otimes x\boo+ \phi\left(x\boi\right) \otimes x\boo+\psi\left(x\boi\right)\otimes x\boo+P\left(x\boi\right)\otimes x\boo\\
&&+\Delta_{H}\left(x\boo\right) \otimes x\bi+ \rho\left(x\boo\right) \otimes x\bi+\gamma\left(x\boo\right)\otimes x\bi+Q\left(x\boo\right)\otimes x\bi\\
&&+\Delta_{A}\left(x\qi\right) \otimes x\qii+\phi\left(x\qi\right) \otimes x\qii+\psi\left(x\qi\right) \otimes x\qii+P\left(x\qi\right) \otimes x\qii\\
&&-a_{1} \otimes \Delta_{A}\left(a_{2}\right)-a_{1} \otimes \phi\left(a_{2}\right)-a_{1} \otimes \psi\left(a_{2}\right)-a_{1} \otimes P\left(a_{2}\right) \\
&&-a_{(-1)} \otimes \Delta_{A}\left(a_{(0)}\right)-a_{(-1)} \otimes \phi\left(a_{(0)}\right)-a_{(-1)} \otimes \psi\left(a_{(0)}\right)-a_{(-1)} \otimes P\left(a_{(0)}\right)\\
&&-a_{(0)} \otimes \Delta_{H}\left(a_{(1)}\right)-a_{(0)} \otimes \rho\left(a_{(1)}\right)-a_{(0)} \otimes \gamma\left(a_{(1)}\right)-a_{(0)} \otimes Q\left(a_{(1)}\right)\\
&&-a\ppi \otimes \Delta_{H}\left(a\pii\right)-a\ppi \otimes \rho\left(a\pii\right)-a\ppi \otimes \gamma\left(a\pii\right)-a\ppi\otimes Q\left(a\pii\right)\\
&&-x_{1} \otimes \Delta_{H}\left(x_{2}\right)-x_{1} \otimes \rho\left(x_{2}\right)-x_{1} \otimes \gamma\left(x_{2}\right)-x_{1} \otimes Q\left(x_{2}\right)\\
&&-x\boi \otimes \Delta_{H}\left(x\boo\right)-x\boi\otimes \rho\left(x\boo\right)-x\boi \otimes \gamma\left(x\boo\right)-x\boi\otimes Q\left(\boo\right)\\
&&-x\boo \otimes \Delta_{A}\left(x\bi\right)-x\boo \otimes \phi\left(x\bi\right)-x\boo\otimes \psi\left(x\bi\right)-x\boo \otimes P\left(x\bi\right)\\
&&-x\qi \otimes \Delta_{A}\left(x\qii\right)-x\qi \otimes \phi\left(x\qii\right)-x\qi \otimes \psi\left(x\qii\right)-x\qi \otimes P\left(x\qii\right)
\end{eqnarray*}
and the right hand side is calculated in the same way. Thus the two sides are equal to each other if and only if  (CCP1)--(CCP8) hold.
\end{proof}

\begin{definition}\label{cocycledmp}
Let $A, H$ be both  anti-flexible algebras and  anti-flexible coalgebras. If  the following conditions hold:
\begin{enumerate}
\item[(CDM1)]  $\phi(ab)+\tau\psi(ba)+\gamma(\theta(a, b))+\tau\rho(\theta(b, a))$\\
$=\theta\left(b, a_{2}\right) \otimes a_{1}+(b \trr a_{(1)}) \otimes a_{(0)}+b\pii \otimes \left(b\ppi \ppr a\right)+b_{(1)}\ot b_{(0)} a$\\
$+\theta\left(a_{1}, b \right) \otimes a_{2}+(a_{(-1)} \trl b) \otimes a_{(0)}+b_{(-1)} \otimes\left( a b_{(0)}\right)+b\ppi \ot (a\ppl b\pii),$


\item[(CDM2)] $\rho(x y)+\tau \gamma( y x)+\psi(\sigma(x, y))+\tau\phi(\sigma(y, x))$\\
$=\sigma\left(y, x_{2}\right) \otimes x_{1}+\left(y \ppr x_{[1]}\right) \otimes x_{[0]}+y_{[1]} \otimes y_{[0]} x+y\qii\otimes (y\qi\trr x)$\\
$+\sigma\left( x_{1}, y\right) \otimes x_{2}+(x_{[-1]}\ppl y)\ot x_{[0]}+y_{[-1]}\ot xy_{[0]}+y\qi\ot (x\trl y\qii), $


\item[(CDM3)] $\Delta_{A}(x \ppr b)+\tau\Delta_{A}(b \ppl x)+Q(x\trl  b)+\tau Q(b \trr x)$ \\
$=\left(b\ppl x_{[0]}\right) \otimes x_{[-1]}+b x\qii \otimes x\qi+b_{2} \otimes\left(b_{1} \ppl x\right)+b\loo\ot \si(b\loi, x)$\\
$+\left(x_{[0]} \ppr b\right)\otimes x_{[1]}+x\qi b \otimes x\qii+b\li \ot (x\ppr b\lii)+b_{(0)} \otimes \sigma\left( x, b_{(1)}\right),$


\item[(CDM4)] $\Delta_{H}(a \trr y)+\tau \Delta_{H}(y \trl a)+P(a\ppl y)+\tau P(y\ppr a)$\\
$= (y\trl a\loo)\ot a_{(-1)}+ ya\pii\ot a\ppi+y_{2} \otimes\left(y_{1}  \trl a\right)+y\poo\ot \tht(y\boi,  a)$\\
$+\left(a_{(0)}\trr y\right)\otimes a_{(1)}+a\ppi y\otimes a\pii+y\li\ot (a\trr y\lii)+y_{[0]} \otimes \theta\left(a , y_{[1]}\right), $


\item[(CDM5)]$\Delta_{H}(\theta(a,b))+\tau\Delta_{H}(\theta(b,a))+P(a b)+\tau P(b a)$\\
$=\theta(b,a_{(0)})\otimes a_{(-1)}+(b\trr a\pii)\otimes a\ppi+b _{(1)} \ot \tht(b\lmoo, a)+b\pii \otimes (b\ppi\trl a)$\\
$+\theta(a_{(0)},b)\otimes a_{(1)} +( a\ppi\trl b)\otimes a\pii+b_{(-1)} \otimes\theta(a, b_{(0)})+b\ppi\ot (a\trr b\pii),$

\item[(CDM6)]$\Delta_{A}(\sigma(x,y))+\tau \Delta_{A}(\sigma(y, x))+Q(x y)+\tau Q(yx)$\\
$=(y\ppr x\qii)\ot x\qi+\sigma(y,x_{[0]})\otimes x_{[-1]}+y\bi \ot \si( y\poo, x)+y\qii\otimes (y\qi\ppl x)$\\
$+\si(x\poo,y )\ot x\bi+(x\qi\ppl y)\otimes x\qii+y_{[-1]}\otimes\sigma(x,y_{[0]})+y\qi\ot (x\ppr y\qii), $

\item[(CDM7)]
 $\phi(x \ppr b)+\tau\psi(b \ppl x)+\gamma(x\trl b)+\tau \rho(b\trr x)$\\
 $=(b\trr x\poo)\ot x\boi+\tht(b, x\qii)\ot x\qi+b_{(1)}\ot (b\loo\ppl x)+b\pii\ot \si(b\ppi, x)$\\
 $+(x_{[0]}\trl b)\otimes x_{[1]}+\tht(x\qi, b)\ot x\qii+b_{(-1)} \otimes (x \ppr b_{(0)})+b\ppi\ot\sigma(x, b\pii), $
\item[(CDM8)]
$\psi(x \ppr b)+\tau\phi(b \ppl x)+\rho(x\trl b)+\tau\gamma(b\trr x)$\\
$=(b\ppl x\lii) \ot x\li+ bx\bi\ot x\poo+b\lii\ot (b\li \trr x)+b\loo\ot b\loi x$\\
$+(x\li \ppr b) \otimes x\lii+x_{[-1]} b \otimes x_{[0]}+b\li\ot (x\trl b\lii)+b\loo\ot xb_{(1)}, $

\item[(CDM9)]$a\li\ot\tht(a\lii, b)+a\loo\ot(a_{(1)}\trl b)+ab\loo\ot b_{(1)}+(a\ppl b\ppi)\ot b\pii$\\
$-b\li\ot\tht(b\lii, a)-b\loo\ot(b_{(1)}\trl a)-ba\loo\ot a_{(1)}-(b\ppl a\ppi)\ot a\pii$\\
 $=\tau\Big(a\loi\ot a\loo b+a\ppi\ot (a\pii\ppr b)+\tht(a, b\li)\ot b\lii+(a\trr b\loi)\ot b\loo$\\
$-b\loi\ot b\loo a-b\ppi\ot (b\pii\ppr a)-\tht(b, a\li)\ot a\lii-(b\trr a\loi)\ot a\loo\Big)$,

\item[(CDM10)]$(\id-\tau)(a\loi\ot\tht(a\loo, b)+a\ppi\ot (a\pii\trl b)+\tht(a, b\loo)\ot b_{(1)}+(a\trr b\ppi)\ot b\pii)$\\
$=(\id-\tau)(b\loi\ot\tht(b\loo, a)+b\ppi\ot (b\pii\trl a)+\tht(b, a\loo)\ot a_{(1)}+(b\trr a\ppi)\ot a\pii)$,

\item[(CDM11)]$(\id-\tau)(x\li\ot(x\lii\trl b)+x\boo\ot\tht(x\bi, b)+(x\trl b\loo)\ot b_{(1)}+xb\ppi\ot x\pii)$\\
$=(\id-\tau)(b\loi\ot(b\loo\trr x)+b\ppi\ot b\pii x+(b\trr x\li)\ot x\lii+\tht(b, x\boi)\ot x\boo)$,

\item[(CDM12)]$(\id-\tau)(x\boi\ot(x\boo\ppr b)+x\qi\ot x\qii b+(x\ppr b\li)\ot b\lii+\si(x, b\loi)\ot b\loo)$\\
$=(\id-\tau)(b\li\ot (b\lii \ppl x)+b\loo\ot\si(b_{(1)}, x)+(b\ppl x\boo)\ot x\bi+bx\qi\ot x\qii)$,

\item[(CDM13)]$x\li\ot (x\lii \ppr b)+x\boo\ot x\bi b+(x\trl b\li)\ot b\lii+xb\loi\ot b\loo$\\
$-b\loi\ot (b\loo\ppl x)-b\ppi\ot\si(b\pii, x)-(b\trr x\boo)\ot x\bi-\tht(b, x\qi)\ot x\qii$\\
$=\tau\Big(x\boi\ot(x\boo\trl b)+x\qi\ot\tht(x\qii, b)+(x\ppr b\loo)\ot b_{(1)}+\si(x, b\ppi)\ot b\pii$\\
$-b\li\ot(b\lii\trr x)-b\loo\ot b_{(1)}x-(b\ppl x\li)\ot x\lii-bx\boi\ot x\boo\Big)$,

\item[(CDM14)]$x\li\ot\si(x\lii, y)+x\boo\ot(x\bi\ppl y)+xy\boo\ot y\bi+(x\trl y\qi)\ot y\qii$\\
$-y\li\ot\si(y\lii, x)-y\boo\ot(y\bi\ppl x)-yx\boo\ot x\bi-(y\trl x\qi)\ot x\qii$\\
$=\tau\Big(x\boi\ot x\boo y+x\qi\ot(x\qii\trr y)+\si(x, y\li)\ot y\lii+(x\ppr y\boi)\ot y\boo$\\
$-y\boi\ot y\boo x-y\qi\ot(y\qii\trr x)-\si(y, x\li)\ot x\lii-(y\ppr x\boi)\ot x\boo\Big)$,

\item[(CDM15)]$(\id-\tau)(x\boi\ot \si(x\boo, y)+x\qi\ot(x\qii\ppl y)+\si(x, y\boo)\ot y\bi+(x\ppr y\qi)\ot y\qii)$\\
$=(\id-\tau)(y\boi\ot \si(y\boo, x)+y\qi\ot(y\qii\ppl x)+\si(y, x\boo)\ot x\bi+(y\ppr x\qi)\ot x\qii)$,
\end{enumerate}
\noindent then $(A, H)$ is called a \emph{cocycle double matched pair}.
\end{definition}

\begin{definition}\label{cocycle-braided}
(i) A \emph{cocycle braided  anti-flexible bialgebra} $A$ is simultaneously a cocycle anti-flexible algebra $(A, \, \theta)$ and  a cycle anti-flexible coalgebra $(A, \,  Q)$ satisfying the  conditions
\begin{enumerate}
\item[(CBB1)] $\Delta_{A}(ab)+\tau\Delta_{A}(ba)+Q\theta(a,b)+\tau Q\theta(b,a)$\\
$=b a_{2}\otimes a_{1} + b_{2}\otimes b_{1}a + a_{1} b \otimes a_{2}+b_{1} \otimes a b_{2}$\\
$ +(b\ppl a_{(1)}) \otimes a_{(0)}+b_{(0)} \otimes(b_{(-1)} \ppr a )+(a_{(-1)}\ppr b) \otimes a_{(0)}+b_{(0)} \otimes (a\ppl b_{(1)}), $

\item[(CBB2)] $(\id-\tau)\left( a\li\ot a\lii b+ab\li\ot b\lii-b\li\ot b\lii a-ba\li\ot a\lii\right)$\\
$+(\id-\tau)(a\loo\ot (a_{(1)}\ppr b)+(a\ppl b\loi)\ot b\loo-b\loo\ot (b_{(1)}\ppr a)-(b\ppl a\loi)\ot a\loo)=0. $
\end{enumerate}
(ii) A \emph{cocycle braided  anti-flexible bialgebra} $H$ is simultaneously a cocycle  anti-flexible algebra $(H, \, \sigma)$ and a cycle  anti-flexible coalgebra $(H, \,  P)$ satisfying the conditions
\begin{enumerate}
\item[(CBB3)] $\Delta_{H}(xy)+\tau\Delta_{H}(y x)+P\sigma(x,y)+\tau P\sigma(y,x)$\\
$=y x_{2}\otimes x_{1} + y_{2}\otimes y_{1}x + x_{1} y \otimes x_{2}+y_{1} \otimes x y_{2}$\\
$ +(y\trl x_{[1]}) \otimes x_{[0]}+y_{[0]} \otimes(y_{[-1]} \trr x )+(x_{[-1]}\trr y) \otimes x_{[0]}+y_{[0]} \otimes (x\trl y_{[1]}), $

\item[(CBB4)]
$(\id-\tau)(x\li\ot x\lii y+xy\li\ot y\lii-y\li\ot y\lii x-yx\li\ot x\lii)$\\
$+(\id-\tau)(x\boo\ot(x\bi\trr y)+(x\trl y\boi)\ot y\boo-y\boo\ot(y\bi\trr x)-(y\trl x\boi)\ot x\boo)=0$.
\end{enumerate}
\end{definition}

The next theorem says that we can obtain an ordinary anti-flexible bialgebra from two cocycle braided anti-flexible bialgebras.
\begin{theorem}\label{main2}
Let $A$, $H$ be  cocycle braided anti-flexible bialgebras, $(A, H)$ be a cocycle cross product system and a cycle cross coproduct system.
Then the cocycle cross product   algebra and cycle cross coproduct   coalgebra fit together to become an ordinary
anti-flexible bialgebra if and only if $(A, H)$ forms a cocycle double matched pair. We will call it the cocycle bicrossproduct anti-flexible bialgebra and denote it by $A^{P}_{\sigma}\# {}^{Q}_{\theta}H$.
\end{theorem}

The proof is by direct computations, so we omit the details.

\section{Extending structures for   anti-flexible bialgebras}
In this section, we will study the extending problem for     anti-flexible bialgebras.
We will find some special cases when the braided anti-flexible bialgebra is reduced into an ordinary    anti-flexible bialgebra.
It is proved that the extending problem can be solved by using of the non-abelian cohomology theory based on cocycle bicrossedproduct for braided anti-flexible bialgebras in last section.

\subsection{Extending structures for anti-flexible algebras (coalgebras)}
First we are going to study extending problem for anti-flexible algebras (coalgebras).

There are two cases for $A$ to be an anti-flexible algebra in the cocycle cross product system defined in last section, see condition (CC6). The first case is when we let $\ppr$, $\ppl$ to be trivial and $\theta\neq 0$,  then from condition (CP1) we get $\si(x, \theta(a, b))+\si(\theta(a, b), x)=0$. Since $\theta\neq 0$ we assume $\sigma=0$ for simplicity, thus  we obtain the following type $(a1)$  unified product for anti-flexible algebras.

\begin{lemma}
Let ${A}$ be an anti-flexible algebra and $V$ a vector space. An extending datum of ${A}$ by $V$ of type (a1)  is  $\Omega^{(1)}({A},V)=(\trr, \, \trl, \, \theta, \, \cdot)$ consisting of bilinear maps
\begin{eqnarray*}
\trr: A\otimes V \to V,\quad \trl: V\otimes A \to V,\quad\theta: A\ot A\to V,\quad\cdot: V\ot V\to V.
\end{eqnarray*}
Denote by $A_{}\#_{\theta}V$ the vector space $E={A}\oplus V$ together with the multiplication given by
\begin{eqnarray}
(a, x)(b, y)=\big(ab, \, xy+x\trl b+a\trr y+\theta(a, b)\big).
\end{eqnarray}
Then $A_{}\# {}_{\theta}V$ is an anti-flexible algebra if and only if the following compatibility conditions hold for all $a$, $b\in {A}$, $x$, $y$, $z\in V$:
\begin{enumerate}
\item[(A1)]$(x y) \trl a+a \trr (yx)=x(y \trl a)+(a \trr y)x$,

\item[(A2)]$ (x\trl a) y-x(a\trr y)=(y\trl a)x-y(a\trr x)$,

\item[(A3)] $(ab) \trr x+\theta(a, b) x-a\trr (b\trr x)=(x\trl b)\trl a -x\theta(b,a)-x\trl(ba)$,

\item[(A4)] $(a\trr x)\trl b -a\trr(x\trl b)=(b\trr x)\trl a -b\trr(x\trl a) $,

\item[(A5)]$\theta(a b, c)-\theta(a, b c)+\theta(a, b) \triangleleft c-a\trr \theta(b, c)=\theta(c b, a)-\theta(c, b a)+\theta(c, b) \trl a-c\trr \theta(b, a),$

\item[(A6)] $(xy)z-x(yz)=(zy)x-z(yx)$,
\end{enumerate}
\end{lemma}
Note that (A1)--(A4)  are deduced from (CP5)--(CP8) and by (A6)  we obtain that $V$ is an anti-flexible algebra. Furthermore, $V$ is in fact an anti-flexible subalgebra of $A_{}\#_{\theta}V$  but $A$ is not although $A$ is itself an anti-flexible algebra.

Denote also the set of all  algebraic extending datum of ${A}$ by $V$ of type (a1)  by $\Omega^{(1)}({A},V)$ by abuse of notations.


In the following, we always assume that $A$ is a subspace of a vector space $E$, there exists a projection map $p: E \to{A}$ such that $p(a) = a$, for all $a \in {A}$.
Then the kernel space $V := \ker(p)$ is also a subspace of $E$ and a complement of ${A}$ in $E$.

\begin{lemma}\label{lem:33-1}
Let ${A}$ be an anti-flexible algebra and $E$ a vector space containing ${A}$ as a subspace.
Suppose that there is an anti-flexible algebra structure on $E$ such that $V$ is an anti-flexible  subalgebra of $E$
and the canonical projection map $p: E\to A$ is an anti-flexible algebra homomorphism.
Then there exists an algebraic extending datum $\Omega^{(1)}({A},V)$ of ${A}$ by $V$ such that
$E\cong A_{}\#_{\theta}V$.
\end{lemma}

\begin{proof}
Since $V$ is a  subalgebra of $E$, we have $x\cdot_E y\in V$ for all $x, y\in V$.
We define the extending datum of ${A}$ through $V$ by the following formulas:
\begin{eqnarray*}
\trl: V\otimes {A} \to V, \qquad {x} \triangleleft {a} &:=&{x}\cdot_E {a},\\
\theta: A\otimes A \to V, \qquad \theta(a,b) &:=&a\cdot_E b-p \bigl(a\cdot_E b\bigl),\\
{\cdot_V}: V \otimes V \to V, \qquad {x}\cdot_V {y}&:=& {x}\cdot_E{y} .
\end{eqnarray*}
for all $a , b\in {A}$ and $x, y\in V$. It is easy to see that the above maps are  well defined and
$\Omega^{(1)}({A}, V)$ is an extending system of ${A}$ through $V$ and
\begin{eqnarray*}
\varphi:A_{}\#_{\theta}V\to E, \qquad \varphi(a, x) := a+x
\end{eqnarray*}
is an isomorphism of  anti-flexible algebras.
\end{proof}

\begin{lemma}
Let $\Omega^{(1)}(A, V) = \bigl(\trr,  \trl, \theta, \cdot \bigl)$ and $\Omega'^{(1)}(A, V) = \bigl(\trr ', \, \trl', \theta', \cdot ' \bigl)$
be two extending datums of ${A}$ by $V$ of type (a1) and $A_{}\#_{\theta} V$, $A_{}\#_{\theta'} V$ be the corresponding unified products. Then there exists a bijection between the set of all homomorphisms of anti-flexible algebras $\varphi: A_{\theta}\#_{\trr, \trl} V\to A_{\theta'}\#_{\trr', \trl'} V$ whose restriction on ${A}$ is the identity map and the set of pairs $(r, s)$, where $r: V\rightarrow {A}$ and $s: V\rightarrow V$ are two linear maps satisfying
\begin{eqnarray}
&&{r}(x\trl a)={r}(x)\cdot' a,\\
&&{r}(a\trr x)= a\cdot'{r}(x),\\
&&a\cdot' b=ab+r\theta(a, b),\\
&&{r}(xy)={r}(x)\cdot' {r}(y),\\
&&{s}(x)\trl' a+\theta'(r(x), a)={s}(x\trl a),\\
&& a\trr'{s}(y)+\theta'(a, r(y) )={s}(a\trr y),\\
&&\theta'(a, b)=s\theta(a, b),\\
&&{s}(xy)={s}(x)\cdot' {s}(y)+{s}(x)\trl'{r}(y)+{r}(x)\trr'{s}(y)+\theta'(r(x), r(y)).
\end{eqnarray}
for all $a\in{A}$ and $x$, $y\in V$.

Under the above bijection the homomorphism of anti-flexible  algebras $\varphi=\varphi_{r, s}: A_{}\#_{\theta}V\to A_{}\#_{\theta'} V$ to $(r, s)$ is given  by $\varphi(a, x)=(a+r(x), s(x))$ for all $a\in {A}$ and $x\in V$. Moreover, $\varphi=\varphi_{r, s}$ is an isomorphism if and only if $s: V\rightarrow V$ is a linear isomorphism.
\end{lemma}

\begin{proof}
Let $\varphi: A_{}\#_{\theta}V\to A_{}\#_{\theta'} V$  be a homomorphism  whose restriction on ${A}$ is the identity map. Then $\varphi$ is determined by two linear maps $r: V\rightarrow {A}$ and $s: V\rightarrow V$ such that
$\varphi(a, x)=(a+r(x), s(x))$ for all $a\in {A}$ and $x\in V$.
In fact, we have to show
$$\varphi((a, x)(b, y))=\varphi(a, x)\cdot'\varphi(b, y).$$
The left hand side is equal to
\begin{eqnarray*}
&&\varphi((a, x)(b, y))\\
&=&\varphi\left({ab},\,  x\trl b+y\trl a+{xy}+\theta(a, b)\right)\\
&=&\big({ab}+ r(x\trl b)+r(y\trl a)+r({xy})+r\theta(a, b),\\
&&\qquad\quad s(x\trl b)+s(y\trl a)+s({xy})+s\theta(a, b)\big),
\end{eqnarray*}
and the right hand side is equal to
\begin{eqnarray*}
&&\varphi(a, x)\cdot' \varphi(b, y)\\
&=&(a+r(x), s(x))\cdot'  (b+r(y), s(y))\\
&=&\big((a+r(x))\cdot' (b+r(y)),  s(x)\trl'(b+r(y))+(a+r(x))\trr's(y)\\
&&\qquad\qquad +s(x)\cdot' s(y)+\theta'(a+r(x), b+r(y))\big).
\end{eqnarray*}
Thus $\varphi$ is a homomorphism of anti-flexible algebras if and only if the above conditions hold.
\end{proof}

The second case is when $\theta=0$,  we obtain the following type (a2)  unified product for  anti-flexible algebras.

%
%
%

\begin{theorem}
Let $A$ be an anti-flexible algebra and $V$ a  vector space. An \textit{extending
datum of $A$ through $V$} of type (a2)  is a system $\Omega^{(2)}(A, V) =
\bigl(\triangleleft, \, \triangleright, \, \leftharpoonup, \,
\rightharpoonup, \, \sigma, \, \cdot \bigl)$ consisting of linear maps
\begin{eqnarray*}
&&\ppr: V \otimes A \to A, \quad \ppl: A \otimes V \to A, \quad\triangleleft : V \otimes A \to V, \quad \trr: A \otimes V \to V \\
&& \sigma: V\otimes V \to A, \quad \cdot \, : V\otimes V \to V
\end{eqnarray*}
Denote by $A{}_{\sigma}\# {}_{}H$ the vector space $E={A}\oplus V$ together with the multiplication
\begin{align}
(a, x)(b, y)=\big(ab+x\ppr b+a\ppl y+\sigma(x, y), \, xy+x\trl b+a\trr y\big).
\end{align}
Then $A{}_{\sigma}\# {}_{}H$  is an anti-flexible  algebra if and only if the following compatibility conditions hold for all $a, b, c\in A$,
$x, y, z\in V$:
\begin{enumerate}
\item[(B1)] $x \ppr(a b)+(ba)\ppl x=(x\ppr a) b+b(a\ppl x)+(x\trl a) \ppr b+b\ppl(a\trr x)$,

\item[(B2)] $(a\ppl x)b+(a\trr x)\ppr b -a(x\ppr b)-a\ppl(x\trl b)$\\
$=(b\trr x)\ppr a+(b\ppl x)a-b(x\ppr a)-b\ppl (x\trl a)$,

\item[(B3)] $(x y) \ppr a+\sigma(x, y) a -x\ppr(y\ppr a)-\sigma(x, y\trl a)$\\
$=(a\ppl y)\ppl x-a\ppl (yx)+\sigma(a\trr y, x)-a\sigma(y, x)$,

\item[(B4)] $(x\ppr a)\ppl y+\sigma(x\trl a, y)-x\ppr (a\ppl y)-\sigma(x, a\trr y)$\\
$=(y\ppr a)\ppl x+\sigma(y\trl a, x)-y\ppr(a\ppl x)-\sigma(y, a\trr x)$,

\item[(B5)]$(x y) \trl a+a \trr (yx)=x(y \trl a)+(a \trr y)x+x\trl(y\ppr a)+(a\ppl y)\trr x$,

\item[(B6)]$ (x\trl a) y+(x\ppr a) \trr y-x(a\trr y)-x\trl(a\ppl y)$\\
$=(y\trl a)x+(y\ppr a)\trr x-y(a\trr x)-y\trl(a\ppl x)$,

\item[(B7)] $(ab) \trr x -a\trr (b\trr x)=(x\trl b)\trl a -x\trl(ba)$,

\item[(B8)] $(a\trr x)\trl b -a\trr(x\trl b)=(b\trr x)\trl a -b\trr(x\trl a) $,

\item[(B9)] $\sigma(x y, z)-{\sigma}(x, y z)+\sigma(x, y)\ppl z-x \ppr \sigma(y, z)\\
=\sigma(z y, x)-{\sigma}(z, y x)+\sigma(z, y)\ppl x-z \ppr \sigma(y, x), $

\item[(B10)] $(x y) z-x(y z)+\sigma(x, y) \trr z-x \triangleleft \sigma(y, z)=(z y) x-z(y x)+\sigma(z, y) \trr x-z \trl \sigma(y, x)$.
\end{enumerate}
\end{theorem}

Similar to Theorem \ref{lem:33-1},  we obtain.
\begin{theorem}
Let $A$ be an  anti-flexible algebra, $E$ a  vector space containing $A$ as a  subspace.
If there is an  anti-flexible algebra structure on $E$ such that $A$ is a  anti-flexible subalgebra of $E$. Then there exists an algebraic extending structure $\Omega^{(2)}(A, V) = \bigl(\triangleleft, \, \triangleright, \,
\leftharpoonup, \, \rightharpoonup, \, \sigma, \, \cdot \bigl)$ of $A$ through $V$ such that there is an isomorphism of anti-flexible algebras $E\cong A_{\sigma}\#_{}H$.
\end{theorem}

\begin{lemma}
Let $\Omega^{(2)}(A, V) = \bigl(\trr, \trl, \leftharpoonup,  \rightharpoonup,  \sigma,  \cdot \bigl)$ and
$\Omega'^{(2)}(A, V) = \bigl(\trr', \trl ', \leftharpoonup ',  \rightharpoonup ', \sigma ', \cdot ' \bigl)$
be two  algebraic extending structures of $A$ through $V$ and $A{}_{\sigma}\#_{}V$, $A{}_{\sigma'}\#_{}  V$ the  associated unified
products. Then there exists a bijection between the set of all homomorphisms of anti-flexible algebras $\psi: A{}_{\sigma}\#_{}V\to A{}_{\sigma'}\#_{}  V$which
stabilize $A$ and the set of pairs $(r, s)$, where $r: V \to
A$, $s: V \to V$ are linear maps satisfying the following
compatibility conditions for all $x \in A$, $u$, $v \in V$:
\begin{enumerate}
\item[(M1)] $r(x \cdot y) = r(x)\cdot'r(y) + \sigma ' (s(x), s(y)) - \sigma(x, y) + r(x) \ppl' s(y) + s(x) \ppr' r(y)$,
\item[(M2)] $s(x \cdot y) = r(x) \trr ' s(y) + s(x)\trl ' r(y) + s(x) \cdot ' s(y)$,
 \item[(M3)] $r(x\trl  {a}) = r(x)\cdot' {a} - x \ppr {a} + s(x) \ppr' {a}$,
  \item[(M5)] $r({a} \trr x) = {a}\cdot' r(x) - {a}\ppl x + {a} \ppl' s(x)$,
 \item[(M4)] $s(x\trl {a}) = s(x)\trl' {a}$,
 \item[(M6)] $s({a}\trr x) = {a} \trr' s(x)$.
\end{enumerate}
Under the above bijection the homomorphism of anti-flexible algebras $\varphi =\varphi _{(r, s)}: A_{\sigma}\# {}_{}H \to A_{\sigma'}\# {}_{}H$ corresponding to
$(r, s)$ is given for all $a\in A$ and $x \in V$ by:
$$\varphi(a, x) = (a + r(x), s(x))$$
Moreover, $\varphi  = \varphi _{(r, s)}$ is an isomorphism if and only if $s: V \to V$ is an isomorphism linear map.
\end{lemma}

Let ${A}$ be an anti-flexible  algebra and $V$ a vector space. Two algebraic extending systems $\Omega^{(i)}({A}, V)$ and ${\Omega'^{(i)}}({A}, V)$  are called equivalent if $\varphi_{r, s}$ is an isomorphism.  We denote it by $\Omega^{(i)}({A}, V)\equiv{\Omega'^{(i)}}({A}, V)$.
From the above lemmas, we obtain the following result.

\begin{theorem}\label{thm3-1}
Let ${A}$ be an anti-flexible algebra, $E$ a vector space containing ${A}$ as a subspace and
$V$ be a complement of ${A}$ in $E$.
Denote $\mathcal{HA}(V, {A}):=\Omega^{(1)}({A}, V)\sqcup \Omega^{(2)}({A}, V) /\equiv$. Then the map
\begin{eqnarray}
\notag&&\Psi: \mathcal{HA}(V, {A})\rightarrow Extd(E, {A}),\\
&&\overline{\Omega^{(1)}({A}, V)}\mapsto A_{}\#_{\theta} V,\quad \overline{\Omega^{(2)}({A}, V)}\mapsto A_{\sigma}\# {}_{} V
\end{eqnarray}
is bijective, where $\overline{\Omega^{(i)}({A}, V)}$ is the equivalence class of $\Omega^{(i)}({A}, V)$ under $\equiv$.
\end{theorem}

Next we consider the  anti-flexible coalgebra structures on $E=A^{P}\# {}^{Q}V$.

There are two cases for $(A, \, \Delta_A)$ to be an  anti-flexible   coalgebra. The first case is  when $Q=0$,  then we obtain the following type (c1) unified coproduct for   anti-flexible coalgebras.
\begin{lemma}\label{cor02}
Let $({A},\Delta_A)$ be an anti-flexible coalgebra and $V$ a vector space.
An  extending datum  of ${A}$ by $V$ of  type (c1) is  $\Omega^{(3)}({A}, V)=(\phi, \,  {\psi}, \, \rho, \, \gamma, \,  P, \,  \Delta_V)$ with  linear maps
\begin{eqnarray*}
&&\phi: A \to V \otimes A, \quad  \psi: A \to A\otimes V,\\
&&\rho: V  \to A\otimes V, \quad  \gamma: V \to V \otimes A,\\
&& {P}: A\rightarrow {V}\otimes {V}, \quad\Delta_V: V\rightarrow V\otimes V.
\end{eqnarray*}
 Denote by $A^{P}\# {}^{} V$ the vector space $E={A}\oplus V$ with the linear map
$\Delta_E: E\rightarrow E\otimes E$ given by
$$\Delta_{E}(a)=(\Delta_{A}+\phi+\psi+P)(a),\quad \Delta_{E}(x)=(\Delta_{V}+\rho+\gamma)(x), $$
that is
$$\Delta_{E}(a)= a\li \ot a\lii+ a\moi \ot a\mo+a\mo\ot a\mi+a\ppi\ot a\pii,$$
$$\Delta_{E}(x)= x\li \ot x\lii+ x\boi \ot x\boo+x\boo \ot x\bi,$$
'Then $A^{P}\# {}^{} V$  is an  anti-flexible   coalgebra with the comultiplication given above if and only if the following compatibility conditions hold:
\begin{enumerate}
\item[(C1)]$\phi\left(a_{1}\right) \otimes a_{2}+\gamma\left(a_{(-1)}\right) \otimes a_{(0)}-a_{(-1)} \otimes \Delta_{A}\left(a_{(0)}\right) \\
    =\tau_{13}\left(\Delta_{A}\left(a_{(0)}\right) \otimes a_{(1)}  -a_{1} \otimes \psi\left(a_{2}\right)-a_{(0)} \otimes \rho\left(a_{(1)}\right)\right)$,

\item[(C2)] $\psi\left(a_{1}\right) \otimes a_{2}+\rho\left(a_{(-1)}\right) \otimes a_{(0)}-a_{1} \otimes \phi(a_{2})-a_{(0)} \otimes \gamma(a_{(1)})\\
     =\tau_{13}\left(\psi\left(a_{1}\right) \otimes a_{2}+\rho\left(a_{(-1)}\right) \otimes a_{(0)}-a_{1} \otimes \phi(a_{2})-a_{(0)} \otimes \gamma(a_{(1)})\right)$,
\item[(C3)] $ \Delta_{V}\left(a_{(-1)}\right) \otimes a_{(0)}+P\left(a_{1}\right) \otimes a_{2}-a_{(-1)} \otimes \phi\left(a_{(0)}\right)-a\ppi\otimes \gamma\left(a\pii\right) $\\
$=\tau_{13}(\psi\left(a_{(0)}\right) \otimes a_{(1)}+\rho\left(a\ppi\right) \otimes a\pii -a_{(0)}\ot \Delta_{V}\left(a_{(1)}\right)-a_{1}\otimes P\left(a_{2}\right) )$,
\item[(C4)] $ a_{(-1)} \otimes \psi\left(a_{(0)}\right)+a\ppi \otimes \rho\left(a\pii\right)-\phi\left(a_{(0)}\right) \otimes a_{(1)}-\gamma\left(a\ppi\right) \otimes a\pii $\\
    $=\tau_{13}\left(  a_{(-1)} \otimes \psi\left(a_{(0)}\right)+a\ppi \otimes \rho\left(a\pii\right)-\phi\left(a_{(0)}\right) \otimes a_{(1)}-\gamma\left(a\ppi\right) \otimes a\pii \right)$,

\item[(C5)] $\rho\left(x_{1}\right) \otimes x_{2}+\psi\left(x_{[-1]}\right) \otimes x_{[0]}-x_{[-1]} \otimes \Delta_{V}\left( x_{[0]}\right)   \\
    =\tau_{13}\left(\Delta_{V}(x\boo)\ot x\bi -x\boo\ot \phi(x\bi)-x_1\ot \gamma(x_2) \right)$,

\item[(C6)] $\gamma\left(x_{1}\right) \otimes x_{2}+\phi\left(x_{[-1]}\right) \otimes x_{[0]}-x_{[0]} \otimes \psi\left(x_{[1]}\right)-x_{1} \otimes \rho(x_{2})\\
=\tau_{13}\left(\gamma\left(x_{1}\right) \otimes x_{2}+\phi\left(x_{[-1]}\right) \otimes x_{[0]}-x_{[0]} \otimes \psi\left(x_{[1]}\right)-x_{1} \otimes \rho(x_{2})\right)$,

\item[(C7)] $\Delta_{A}\left(x_{[-1]}\right) \otimes x_{[0]}  -x_{[-1]} \otimes \rho\left(x_{[0]}\right)
    =\tau_{13} ( \gamma(x\boo)\ot x\bi  -x\boo\ot\Delta_A(x\bi))$,

\item[(C8)] $\rho(x\boo)\ot x\bi -x\boi\ot\gamma(x\boo)
=\tau_{13}\left(\rho(x\boo)\ot x\bi -x\boi\ot\gamma(x\boo) \right)$.
\end{enumerate}
\begin{enumerate}
\item[(C9)]  $\Delta_V(a\ppi)\ot a\pii-a\ppi\ot \Delta_V(a\pii)+P(a\lmoo)\ot a\mi-a\lmoi\ot P(a\lmoo)\\
=\tau_{13}\left(\Delta_V(a\ppi)\ot a\pii-a\ppi\ot \Delta_V(a\pii)+P(a\lmoo)\ot a\mi-a\lmoi\ot P(a\lmoo)\right)$.
\end{enumerate}
\begin{enumerate}
\item[(C10)] $\Delta_V(x\li)\ot x\lii-x_1\ot \Delta_V(x_2)+ P(x\boi) \ot x\boo-x\boo\ot P(x\bi)\\
=\tau_{13} \left(\Delta_V(x\li)\ot x\lii-x_1\ot \Delta_V(x_2)+ P(x\boi) \ot x\boo-x\boo\ot P(x\bi) \right)$.
\end{enumerate}
\end{lemma}
Denote also  the set of all  coalgebraic extending datum of ${A}$ by $V$ of type (c1) by $\Omega^{(3)}({A}, V)$ by abuse of notations.

\begin{lemma}\label{lem:33-3}
Let $({A},\Delta_A)$ be an anti-flexible  coalgebra and $E$ a vector space containing ${A}$ as a subspace. Suppose that there is an anti-flexible coalgebra structure $(E,\Delta_E)$ on $E$ such that  $p: E\to {A}$ is an anti-flexible    coalgebra homomorphism. Then there exists a  coalgebraic extending system $\Omega^{(3)}({A}, V)$ of $({A},\Delta_A)$ by $V$ such that $(E,\Delta_E)\cong A^{P}\# {}^{} V$.
\end{lemma}

\begin{proof}
Let $p: E\to {A}$ and $\pi: E\to V$ be the projection map and $V=\ker({p})$.
Then the extending datum of $({A},\Delta_A)$ by $V$ is defined as follows:
\begin{eqnarray*}
&&{\phi}: A\rightarrow V\ot {A},~~~~{\phi}(a)=(\pi\otimes {p})\Delta_E(a),\\
&&{\psi}: A\rightarrow A\ot V,~~~~{\psi}(a)=({p}\otimes \pi)\Delta_E(a),\\
&&{\rho}: V\rightarrow A\ot V,~~~~{\rho}(x)=({p}\otimes \pi)\Delta_E(x),\\
&&{\gamma}: V\rightarrow V\ot {A},~~~~{\gamma}(x)=(\pi\otimes {p})\Delta_E(x),\\
&&\Delta_V: V\rightarrow V\otimes V,~~~~\Delta_V(x)=(\pi\otimes \pi)\Delta_E(x),\\
&&Q: V\rightarrow {A}\otimes {A},~~~~Q(x)=({p}\otimes {p})\Delta_E(x)\\
&&P: A\rightarrow {V}\otimes {V},~~~~P(a)=({\pi}\otimes {\pi})\Delta_E(a).
\end{eqnarray*}
One check that  $\varphi: A^{P}\# {}^{} V\to E$ given by $\varphi(a,x)=a+x$ for all $a\in A, x\in V$ is an anti-flexible    coalgebra isomorphism.
\end{proof}

\begin{lemma}\label{lem-c1}
Let $\Omega^{(3)}({A}, V)=(\phi, {\psi}, \rho, \gamma, P, \Delta_V)$ and ${\Omega'^{(3)}}({A}, V)=(\phi', {\psi'}, \rho', \gamma',  P', \Delta'_V)$ be two  coalgebraic extending datums of $({A}, \Delta_A)$ by $V$. Then there exists a bijection between the set of   anti-flexible   coalgebra homomorphisms $\varphi: A^{P}\# {}^{} V\rightarrow A^{P'}\# {}^{} V$ whose restriction on ${A}$ is the identity map and the set of pairs $(r, s)$, where $r: V\rightarrow {A}$ and $s: V\rightarrow V$ are two linear maps satisfying
\begin{eqnarray}
\label{comorph11}&&P'(a)=s(a\ppi)\ot s(a\pii),\\
\label{comorph121}&&\phi'(a)={s}(a\lmoi)\ot a\lmo+s(a\ppi)\ot r(a\pii),\\
\label{comorph122}&&\psi'(a)=a\lmo\ot {s}(a\mi) +r(a\ppi)\ot s(a\pii),\\
\label{comorph13}&&\Delta'_A(a)=\Delta_A(a)+{r}(a\lmoi)\ot a\lmo+a\lmo\ot {r}(a\mi)+r(a\ppi)\ot r(a\pii)\\
\label{comorph21}&&\Delta_V'({s}(x))+P'(r(x))=({s}\otimes {s})\Delta_V(x),\\
\label{comorph221}&&{\rho}'({s}(x))+\psi'(r(x))=r(x\li)\ot s(x\lii)+x\boi\ot s(x\boo),\\
\label{comorph222}&&{\gamma}'({s}(x))+\phi'(r(x))=s(x\li)\ot r(x\lii)+s(x\boo)\ot x\bi,\\
\label{comorph23}&&\Delta'_A({r}(x))=r(x\li)\ot r(x\lii)+x\boi\ot r(x\boo)+r(x\boo)\ot x\bi.
\end{eqnarray}
Under the above bijection the anti-flexible  coalgebra homomorphism $\varphi=\varphi_{r, s}: A^{P}\# {}^{} V\rightarrow A^{P'}\# {}^{} V$ to $(r, s)$ is given by $\varphi(a, x)=(a+r(x), s(x))$ for all $a\in {A}$ and $x\in V$. Moreover, $\varphi=\varphi_{r, s}$ is an isomorphism if and only if $s: V\rightarrow V$ is a linear isomorphism.
\end{lemma}
\begin{proof}
Let $\varphi: A^{P}\# {}^{} V\rightarrow A^{P'}\# {}^{} V$  be a  coalgebra homomorphism  whose restriction on ${A}$ is the identity map. Then $\varphi$ is determined by two linear maps $r: V\rightarrow {A}$ and $s: V\rightarrow V$ such that
$\varphi(a, x)=(a+r(x), s(x))$ for all $a\in {A}$ and $x\in V$. We will prove that
$\varphi$ is a homomorphism of  anti-flexible  coalgebras if and only if the above condtions hold.
First  it is easy to see that  $\Delta'_E\varphi(a)=(\varphi\otimes \varphi)\Delta_E(a)$ for all $a\in {A}$.
\begin{eqnarray*}
\Delta'_E\varphi(a)&=&\Delta'_E(a)=\Delta'_A(a)+\phi'(a)+\psi'(a)+P'(a),
\end{eqnarray*}
and
\begin{eqnarray*}
&&(\varphi\otimes \varphi)\Delta_E(a)\\
&=&(\varphi\otimes \varphi)\left(\Delta_A(a)+\phi(a)+\psi(a)+P(a)\right)\\
&=&\Delta_A(a)+{r}(a\lmoi)\ot a\lmo+{s}(a\lmoi)\ot a\lmo+a\lmo\ot {r}(a\mi) +a\lmo\ot {s}(a\mi)\\
&&+r(a\ppi)\ot r(a\pii)+r(a\ppi)\ot s(a\pii)+s(a\ppi)\ot r(a\pii)+s(a\ppi)\ot s(a\pii).
\end{eqnarray*}
Thus we obtain that $\Delta'_E\varphi(a)=(\varphi\otimes \varphi)\Delta_E(a)$  if and only if the conditions \eqref{comorph11}, \eqref{comorph121}, \eqref{comorph122} and \eqref{comorph13} hold.
Then we consider that $\Delta'_E\varphi(x)=(\varphi\otimes \varphi)\Delta_E(x)$ for all $x\in V$.
\begin{eqnarray*}
\Delta'_E\varphi(x)&=&\Delta'_E({r}(x), {s}(x))=\Delta'_E({r}(x))+\Delta'_E({s}(x))\\
&=&\Delta'_A({r}(x))+\phi'(r(x))+\psi'(r(x))+P'(r(x))+\Delta'_V({s}(x))+{\rho}'({s}(x))+{\gamma}'({s}(x))),
\end{eqnarray*}
and
\begin{eqnarray*}
&&(\varphi\otimes \varphi)\Delta_E(x)\\
&=&(\varphi\otimes \varphi)(\Delta_V(x)+{\rho}(x)+{\gamma}(x))\\
&=&(\varphi\otimes \varphi)(x\li\ot x\lii+x\boi\ot x\boo+x\boo\ot x\bi)\\
&=&r(x\li)\ot r(x\lii)+r(x\li)\ot s(x\lii)+s(x\li)\ot r(x\lii)+s(x\li)\ot s(x\lii)\\
&&+x\boi\ot r(x\boo)+x\boi\ot s(x\boo)+r(x\boo)\ot x\bi+s(x\boo)\ot x\bi.
\end{eqnarray*}
Thus we obtain that $\Delta'_E\varphi(x)=(\varphi\otimes \varphi)\Delta_E(x)$ if and only if the conditions  \eqref{comorph21},  \eqref{comorph221},  \eqref{comorph222} and \eqref{comorph23}  hold. By definition, we obtain that $\varphi=\varphi_{r,s}$ is an isomorphism if and only if $s: V\rightarrow V$ is a linear isomorphism.
\end{proof}

The second case is $\phi=0$ and  $\psi=0$, and from $(CCP3)$,  we get $P=0$ when $Q\neq 0$. Then we obtain  the following type (c2) unified coproduct for anti-flexible  coalgebras.
\begin{lemma}\label{cor02}
Let $({A},\Delta_A)$ be an anti-flexible    coalgebra and $V$ a vector space.
An  extending datum  of $({A},\Delta_A)$ by $V$ of type (c2)  is  $\Omega^{(4)}({A},V)=(\rho, \gamma, {Q}, \Delta_V)$ with  linear maps
\begin{eqnarray*}
&&\rho: V  \to A\otimes V,\quad  \gamma: V \to V \otimes A,\quad \Delta_{V}: V \to V\otimes V,\quad Q: V \to A\otimes A.
\end{eqnarray*}
 Denote by $A^{}\# {}^{Q} V$ the vector space $E={A}\oplus V$ with the comultiplication
$\Delta_E: E\rightarrow E\otimes E$ given by
\begin{eqnarray}
\Delta_{E}(a)&=&\Delta_{A}(a),\quad \Delta_{E}(x)=(\Delta_{V}+\rho+\gamma+Q)(x), \\
\Delta_{E}(a)&=& a\li \ot a\lii,\quad \Delta_{E}(x)= x\li \ot x\lii+ x\boi \ot x\boo+x\boo \ot x\bi+x\qi\ot x\qii.
\end{eqnarray}
Then $A^{}\# {}^{Q} V$  is an anti-flexible    coalgebra with the comultiplication given above if and only if the following compatibility conditions hold:
\begin{enumerate}
\item[(D1)] $\rho\left(x_{1}\right) \otimes x_{2} -x_{[-1]} \otimes \Delta_{V}\left( x_{[0]}\right)  =\tau_{13}\left(\Delta_{V}(x\boo)\ot x\bi  -x_1\ot \gamma(x_2) \right)$,

\item[(D2)] $\gamma\left(x_{1}\right) \otimes x_{2} -x_{1} \otimes \rho(x_{2})
=\tau_{13}\left(\gamma\left(x_{1}\right) \otimes x_{2} -x_{1} \otimes \rho(x_{2})\right)$,

\item[(D3)] $\Delta_{A}\left(x_{[-1]}\right) \otimes x_{[0]} +Q\left(x_{1}\right) \otimes x_{2}-x_{[-1]} \otimes \rho\left(x_{[0]}\right)  \\
    =\tau_{13}\left( \gamma(x\boo)\ot x\bi  -x\boo\ot\Delta_A(x\bi)-x_1\ot Q(x_2) \right)$,

\item[(D4)] $\rho(x\boo)\ot x\bi -x\boi\ot\gamma(x\boo)
=\tau_{13}(\rho(x\boo)\ot x\bi -x\boi\ot\gamma(x\boo)) $.
\end{enumerate}
\begin{enumerate}
\item[(D5)]  $\Delta_A(x\qi)\ot x\qii-x\qi\ot \Delta_A(x\qii)+Q(x\boo)\ot x\bi-x\boi\ot Q(x\boo)\\
=\tau_{13}\left( \Delta_A(x\qi)\ot x\qii-x\qi\ot \Delta_A(x\qii)+Q(x\boo)\ot x\bi-x\boi\ot Q(x\boo)\right)$.
\end{enumerate}
\begin{enumerate}
\item[(D6)] $\Delta_V(x\li)\ot x\lii-x_1\ot \Delta_V(x_2)
=\tau_{13} \left(\Delta_V(x\li)\ot x\lii-x_1\ot \Delta_V(x_2)  \right)$.
\end{enumerate}
\end{lemma}
Note that in this case $(V, \Delta_V)$ is an anti-flexible  coalgebra.

Denote the set of all  coalgebraic extending datum of ${A}$ by $V$ of type (c2) by $\Omega^{(4)}({A}, V)$.
Similar to the  anti-flexible algebra case,  one  show that any  anti-flexible  coalgebra structure on $E$ containing ${A}$ as an  anti-flexible   subcoalgebra is isomorphic to such an unified coproduct.
\begin{lemma}\label{lem:33-4}
Let $({A}, \Delta_A)$ be an anti-flexible  coalgebra and $E$ a vector space containing ${A}$ as a subspace. Suppose that there is an  anti-flexible coalgebra structure $(E, \Delta_E)$ on $E$ such that  $({A}, \Delta_A)$ is an anti-flexible  subcoalgebra of $E$. Then there exists a  coalgebraic extending system $\Omega^{(4)}({A}, V)$ of $({A}, \Delta_A)$ by $V$ such that $(E, \Delta_E)\cong A^{}\# {}^{Q} V$.
\end{lemma}

\begin{proof}
Let $p: E\to {A}$ and $\pi: E\to V$ be the projection map and $V=ker({p})$.
Then the extending datum of $({A}, \Delta_A)$ by $V$ is defined as follows:
\begin{eqnarray*}
&&{\rho}: V\rightarrow A\ot V,~~~~{\phi}(x)=(p\otimes {\pi})\Delta_E(x),\\
&&{\gamma}: V\rightarrow V\ot {A},~~~~{\phi}(x)=(\pi\otimes {p})\Delta_E(x),\\
&&\Delta_V: V\rightarrow V\otimes V,~~~~\Delta_V(x)=(\pi\otimes \pi)\Delta_E(x),\\
&&Q: V\rightarrow {A}\otimes {A},~~~~Q(x)=({p}\otimes {p})\Delta_E(x).
\end{eqnarray*}
One check that  $\varphi: A^{}\# {}^{Q} V\to E$ given by $\varphi(a, x)=a+x$ for all $a\in A, x\in V$ is an anti-flexible   coalgebra isomorphism.
\end{proof}

\begin{lemma}\label{lem-c2}
Let $\Omega^{(4)}({A}, V)=(\rho, \gamma, {Q}, \Delta_V)$ and ${\Omega'^{(4)}}({A}, V)=(\rho', \gamma', {Q'}, \Delta'_V)$ be two  coalgebraic extending datums of $({A},   \Delta_A)$ by $V$. Then there exists a bijection between the set of   coalgebra homomorphisms $\varphi: A \# {}^{Q} V\rightarrow A \# {}^{Q'} V$ whose restriction on ${A}$ is the identity map and the set of pairs $(r, s)$, where $r: V\rightarrow {A}$ and $s: V\rightarrow V$ are two linear maps satisfying
\begin{eqnarray}
\label{comorph1}&&{\rho}'({s}(x))=r(x\li)\ot s(x\lii)+x\boi\ot s(x\boo),\\
\label{comorph2}&&{\gamma}'({s}(x))=s(x\li)\ot r(x\lii)+s(x\boo)\ot x\bi,\\
\label{comorph3}&&\Delta_V'({s}(x))=({s}\otimes {s})\Delta_V(x)\\
\label{comorph4}&&\Delta'_A({r}(x))+{Q'}({s}(x))=r(x\li)\ot r(x\lii)+x\boi\ot r(x\boo)+r(x\boo)\ot x\bi+{Q}(x).
\end{eqnarray}
Under the above bijection the  coalgebra homomorphism $\varphi=\varphi_{r, s}: A^{ }\# {}^{Q} V\rightarrow A^{ }\# {}^{Q'} V$ to $(r, s)$ is given by $\varphi(a, x)=(a+r(x), s(x))$ for all $a\in {A}$ and $x\in V$. Moreover, $\varphi=\varphi_{r, s}$ is an isomorphism if and only if $s: V\rightarrow V$ is a linear isomorphism.
\end{lemma}
\begin{proof} The proof is similar to the proof of Lemma \ref{lem-c1}.
Let $\varphi: A^{ }\# {}^{Q} V\rightarrow A^{}\# {}^{Q'} V$  be a  coalgebra homomorphism  whose restriction on ${A}$ is the identity map.
First  it is easy to see that  $\Delta'_E\varphi(a)=(\varphi\otimes \varphi)\Delta_E(a)$ for all $a\in {A}$.
Then we consider that $\Delta'_E\varphi(x)=(\varphi\otimes \varphi)\Delta_E(x)$ for all $x\in V$.
\begin{eqnarray*}
\Delta'_E\varphi(x)&=&\Delta'_E({r}(x), {s}(x))=\Delta'_E({r}(x))+\Delta'_E({s}(x))\\
&=&\Delta'_A({r}(x))+\Delta'_V({s}(x))+{\rho}'({s}(x))+{\gamma}'({s}(x))+{Q}'({s}(x)),
\end{eqnarray*}
and
\begin{eqnarray*}
&&(\varphi\otimes \varphi)\Delta_E(x)\\
&=&(\varphi\otimes \varphi)(\Delta_V(x)+{\rho}(x)+{\gamma}(x)+{Q}(x))\\
&=&(\varphi\otimes \varphi)(x\li\ot x\lii+x\boi\ot x\boo+x\boo\ot x\bi+{Q}(x))\\
&=&r(x\li)\ot r(x\lii)+r(x\li)\ot s(x\lii)+s(x\li)\ot r(x\lii)+s(x\li)\ot s(x\lii)\\
&&+x\boi\ot r(x\boo)+x\boi\ot s(x\boo)+r(x\boo)\ot x\bi+s(x\boo)\ot x\bi+{Q}(x).
\end{eqnarray*}
Thus we obtain that $\Delta'_E\varphi(x)=(\varphi\otimes \varphi)\Delta_E(x)$ if and only if the conditions \eqref{comorph1},  \eqref{comorph2},  \eqref{comorph3} and \eqref{comorph4} hold. By definition, we obtain that $\varphi=\varphi_{r,s}$ is an isomorphism if and only if $s: V\rightarrow V$ is a linear isomorphism.
\end{proof}

Let $({A},\Delta_A)$ be an anti-flexible   coalgebra and $V$ a vector space. Two  coalgebraic extending systems $\Omega^{(i)}({A}, V)$ and ${\Omega'^{(i)}}({A}, V)$  are called equivalent if $\varphi_{r,s}$ is an isomorphism.  We denote it by $\Omega^{(i)}({A}, V)\equiv{\Omega'^{(i)}}({A}, V)$.
From the above lemmas, we obtain the following result.
\begin{theorem}\label{thm3-2}
Let $({A},\Delta_A)$ be an anti-flexible   coalgebra, $E$ a vector space containing ${A}$ as a subspace and
$V$ be a ${A}$-complement in $E$. Denote $\mathcal{HC}(V,{A}):=\Omega^{(3)}({A},V)\sqcup\Omega^{(4)}({A},V) /\equiv$. Then the map
\begin{eqnarray*}
&&\Psi: \mathcal{HC}(V,{A})\rightarrow CExtd(E,{A}),\\
&&\overline{\Omega^{(3)}({A},V)}\mapsto A^{P}\# {}^{} V,
 \quad \overline{\Omega^{(4)}({A},V)}\mapsto A^{}\# {}^{Q} V
\end{eqnarray*}
is bijective, where $\overline{\Omega^{(i)}({A},V)}$ is the equivalence class of $\Omega^{(i)}({A}, V)$ under $\equiv$.
\end{theorem}

\subsection{Extending structures for    anti-flexible bialgebras}
Let $(A,\cdot,\Delta_A)$ be an anti-flexible bialgebra. From (CBB1) we have the following two cases.

The first case is that we assume $Q=0$ and $\ppr, \ppl$ to be trivial. Then by the above Theorem \ref{main2}, we obtain the following result.

\begin{theorem}\label{thm-41}
Let $(A, \cdot, \Delta_A)$ be an anti-flexible bialgebra and $V$ a vector space.
An extending datum of ${A}$ by $V$ of type (I) is  $\Omega^{(I)}({A}, V)=(\trr, \trl, \theta, \phi, \psi, P, \rho, \gamma, \cdot_V, \Delta_V)$ consisting of  linear maps
\begin{eqnarray*}
\trr: A\otimes {V}\rightarrow V,~~~~\trl: V\otimes {A}\rightarrow V,~~~~\theta:  A\otimes A \rightarrow {V},~~~{P}: A\rightarrow {V}\otimes {V},~~~~\cdot_V: V\otimes V \rightarrow V,\\
 \phi :A \to V\otimes A, \quad{\psi}: A\to  A\otimes V,~~~~\rho: V\to A\ot V,~~~~\gamma: V\to V\ot A,~~~~\Delta_V: V\rightarrow V\otimes V.
\end{eqnarray*}
Then the unified biproduct $A^{P}_{}\# {}^{}_{\theta}\, V$ with multiplication
\begin{align}
(a, x) (b, y):=(ab, xy+ a\trr y+x\trl b+\theta(a, b))
\end{align}
and comultiplication
\begin{eqnarray}
\Delta_E(a)=\Delta_A(a)+{\phi}(a)+{\psi}(a)+P(a),\quad \Delta_E(x)=\Delta_V(x)+{\rho}(x)+{\gamma}(x)
\end{eqnarray}
forms an anti-flexible bialgebra if and only if $A_{}\# {}_{\theta} V$ forms an anti-flexible  algebra, $A^{P}\# {}^{} \, V$ forms an anti-flexible coalgebra and the following conditions are satisfied:
\begin{enumerate}
\item[(E1)]  $\phi(ab)+\tau\psi(ba)+\gamma(\theta(a, b))+\tau\rho(\theta(b, a))$\\
$=\theta\left(b, a_{2}\right) \otimes a_{1}+(b \trr a_{(1)}) \otimes a_{(0)}+b_{(1)}\ot b_{(0)} a$\\
$+\theta\left(a_{1}, b \right) \otimes a_{2}+(a_{(-1)} \trl b) \otimes a_{(0)}+b_{(-1)} \otimes a b_{(0)},$

\item[(E2)] $\rho(x y)+\tau \gamma( y x)= y_{[1]} \otimes y_{[0]} x+y_{[-1]}\ot xy_{[0]} , $

\item[(E3)] $\Delta_{V}(a \trr y)+\tau \Delta_{V}(y \trl a)$\\
$= (y\trl a\loo)\ot a_{(-1)}+ ya\pii\ot a\ppi+y_{2} \otimes\left(y_{1}  \trl a\right)+y\poo\ot \tht(y\boi,  a)$\\
$+\left(a_{(0)}\trr y\right)\otimes a_{(1)}+a\ppi y\otimes a\pii+y\li\ot (a\trr y\lii)+y_{[0]} \otimes \theta\left(a , y_{[1]}\right), $

\item[(E4)]$\Delta_{V}(\theta(a,b))+\tau\Delta_{V}(\theta(b,a))+P(a b)+\tau P(b a)$\\
$=\theta(b,a_{(0)})\otimes a_{(-1)}+(b\trr a\pii)\otimes a\ppi+b _{(1)} \ot \tht(b\lmoo, a)+b\pii \otimes (b\ppi\trl a)$\\
$+\theta(a_{(0)},b)\otimes a_{(1)} +( a\ppi\trl b)\otimes a\pii+b_{(-1)} \otimes\theta(a, b_{(0)})+b\ppi\ot (a\trr b\pii),$

\item[(E5)]
 $\gamma(x\trl b)+\tau \rho(b\trr x)=(b\trr x\poo)\ot x\boi+(x_{[0]}\trl b)\otimes x_{[1]}, $

\item[(E6)]
$\rho(x\trl b)+\tau\gamma(b\trr x)$\\
$= bx\bi\ot x\poo+b\lii\ot (b\li \trr x)+b\loo\ot b\loi x$\\
$+x_{[-1]} b \otimes x_{[0]}+b\li\ot (x\trl b\lii)+b\loo\ot xb_{(1)}, $

\item[(E7)]$a\li\ot\tht(a\lii, b)+a\loo\ot(a_{(1)}\trl b)+ab\loo\ot b_{(1)}$\\
$-b\li\ot\tht(b\lii, a)-b\loo\ot(b_{(1)}\trl a)-ba\loo\ot a_{(1)})$\\
$=\tau\Big(a\loi\ot a\loo b+\tht(a, b\li)\ot b\lii+(a\trr b\loi)\ot b\loo$\\
$-b\loi\ot b\loo a-\tht(b, a\li)\ot a\lii-(b\trr a\loi)\ot a\loo\Big)$,

\item[(E8)]$(\id-\tau)(a\loi\ot\tht(a\loo, b)+a\ppi\ot (a\pii\trl b)+\tht(a, b\loo)\ot b_{(1)}+(a\trr b\ppi)\ot b\pii)$\\
$=(\id-\tau)(b\loi\ot\tht(b\loo, a)+b\ppi\ot (b\pii\trl a)+\tht(b, a\loo)\ot a_{(1)}+(b\trr a\ppi)\ot a\pii)$,

\item[(E9)]$x\boo\ot x\bi b+(x\trl b\li)\ot b\lii+xb\loi\ot b\loo-(b\trr x\boo)\ot x\bi$\\
$=\tau\Big(x\boi\ot(x\boo\trl b-b\li\ot(b\lii\trr x)-b\loo\ot b_{(1)}x-bx\boi\ot x\boo\Big)$,

\item[(E10)]$(\id-\tau)(x\li\ot(x\lii\trl b)+x\boo\ot\tht(x\bi, b)+(x\trl b\loo)\ot b_{(1)}+xb\ppi\ot x\pii)$\\
$=(\id-\tau)(b\loi\ot(b\loo\trr x)+b\ppi\ot b\pii x+(b\trr x\li)\ot x\lii+\tht(b, x\boi)\ot x\boo)$,

\item[(E12)]$xy\boo\ot y\bi - yx\boo\ot x\bi=\tau\Big(x\boi\ot x\boo y-y\boi\ot y\boo x\Big)$,
\end{enumerate}
\begin{enumerate}
\item[(E13)] $\Delta_{V}(xy)+\tau\Delta_{V}(y x)$\\
$=y x_{2}\otimes x_{1} + y_{2}\otimes y_{1}x + x_{1} y \otimes x_{2}+y_{1} \otimes x y_{2}$\\
$ +(y\trl x_{[1]}) \otimes x_{[0]}+y_{[0]} \otimes(y_{[-1]} \trr x )+(x_{[-1]}\trr y) \otimes x_{[0]}+y_{[0]} \otimes (x\trl y_{[1]}), $

\item[(E14)]
$(\id-\tau)(x\li\ot x\lii y+xy\li\ot y\lii-y\li\ot y\lii x-yx\li\ot x\lii)$\\
$+(\id-\tau)(x\boo\ot(x\bi\trr y)+(x\trl y\boi)\ot y\boo-y\boo\ot(y\bi\trr x)-(y\trl x\boi)\ot x\boo)=0$.
\end{enumerate}
Conversely, any anti-flexible bialgebra structure on $E$ with the canonical projection map $p: E\to A$ both an algebra homomorphism and a   coalgebra homomorphism is of this form.
\end{theorem}
Note that in this case, $(V, \cdot, \Delta_V)$ is a  braided    anti-flexible bialgebra.
Although $(A, \cdot, \Delta_A)$ is not a  sub-bialgebra of $E=A^{P}_{}\# {}^{}_{\theta}\, V$, but it is indeed an anti-flexible bialgebra and a subspace $E$.
Denote the set of all anti-flexible bialgebraic extending datum of type (I) by $\Omega^{(I)}({A}, V)$.

The second case is that we assume $P=0, \theta=0$ and $\phi, \psi$ to be trivial. Then by the above Theorem \ref{main2}, we obtain the following result.

\begin{theorem}\label{thm-42}
Let $A$ be a    anti-flexible bialgebra and $V$ a vector space.
An extending datum of ${A}$ by $V$ of type (II) is  $\Omega^{(II)}({A}, V)=(\ppr, \ppl, \trr, \trl, \sigma, \rho, \gamma, Q,  \cdot_V, \Delta_V)$ consisting of  linear maps
\begin{eqnarray*}
\ppr: V\ot A \to A,~~~~\ppl: A\ot V\to A,~~~~\trl: V\otimes {A}\rightarrow {V},~~~~\trr: A\otimes {V}\rightarrow V,~~~~\sigma:  V\otimes V \rightarrow {A},\\
{\rho}: V\to  A\otimes V,~~~~{\gamma}: V\to  V\otimes A,~~~~{Q}: V\rightarrow {A}\otimes {A},~~~\cdot_V:V\otimes V \rightarrow V,~~~~\Delta_V: V\rightarrow V\otimes V.
\end{eqnarray*}
Then the unified biproduct $A^{}_{\sigma}\# {}^{Q}_{}\, V$ with multiplication
\begin{align}
(a, x)(b, y)=\big(ab+x\ppr b+a\ppl y+\sigma(x, y), \, xy+x\trl b+a\trr y\big).
\end{align}
and comultiplication
\begin{eqnarray}
\Delta_E(a)=\Delta_A(a),\quad \Delta_E(x)=\Delta_V(x)+{\rho}(x)+{\gamma}(x)+Q(x)
\end{eqnarray}
forms an anti-flexible bialgebra if and only if $A_{\sigma}\# {}_{} V$ forms an anti-flexible  algebra, $A^{}\# {}^{Q}V$ forms an anti-flexible    coalgebra and the following conditions are satisfied:
\begin{enumerate}
\item[(F1)] $\rho(x y)+\tau \gamma( y x)$\\
$=\sigma\left(y, x_{2}\right) \otimes x_{1}+\left(y \ppr x_{[1]}\right) \otimes x_{[0]}+y_{[1]} \otimes y_{[0]} x+y\qii\otimes (y\qi\trr x)$\\
$+\sigma\left( x_{1}, y\right) \otimes x_{2}+(x_{[-1]}\ppl y)\ot x_{[0]}+y_{[-1]}\ot xy_{[0]}+y\qi\ot (x\trl y\qii), $

\item[(F2)] $\Delta_{A}(x \ppr b)+\tau\Delta_{A}(b \ppl x)+Q(x\trl  b)+\tau Q(b \trr x)$ \\
$=\left(b\ppl x_{[0]}\right) \otimes x_{[-1]}+b x\qii \otimes x\qi+b_{2} \otimes\left(b_{1} \ppl x\right)$\\
$+\left(x_{[0]} \ppr b\right)\otimes x_{[1]}+x\qi b \otimes x\qii+b\li \ot (x\ppr b\lii)$,

\item[(F3)] $\Delta_{V}(a \trr y)+\tau \Delta_{V}(y \trl a)= y_{2} \otimes\left(y_{1}  \trl a\right)+y\li\ot (a\trr y\lii), $

\item[(F4)]$\Delta_{A}(\sigma(x,y))+\tau \Delta_{A}(\sigma(y, x))+Q(x y)+\tau Q(yx)$\\
$=(y\ppr x\qii)\ot x\qi+\sigma(y,x_{[0]})\otimes x_{[-1]}+y\bi \ot \si( y\poo, x)+y\qii\otimes (y\qi\ppl x)$\\
$+\si(x\poo,y )\ot x\bi+(x\qi\ppl y)\otimes x\qii+y_{[-1]}\otimes\sigma(x,y_{[0]})+y\qi\ot (x\ppr y\qii), $

\item[(F5)]
 $\gamma(x\trl b)+\tau \rho(b\trr x)=(b\trr x\poo)\ot x\boi+(x_{[0]}\trl b)\otimes x_{[1]}, $

\item[(F6)]
$\rho(x\trl b)+\tau\gamma(b\trr x)$\\
$=(b\ppl x\lii) \ot x\li+ bx\bi\ot x\poo+b\lii\ot (b\li \trr x)$\\
$+(x\li \ppr b) \otimes x\lii+x_{[-1]} b \otimes x_{[0]}+b\li\ot (x\trl b\lii), $

\item[(F7)]$  x\li\ot (x\lii \ppr b)+x\boo\ot x\bi b+(x\trl b\li)\ot b\lii-b\trr x\boo)\ot x\bi$\\
$=\tau\Big(x\boi\ot(x\boo\trl b)-b\li\ot(b\lii\trr x)-(b\ppl x\li)\ot x\lii-bx\boi\ot x\boo\Big)$,

\item[(F8)]$(\id-\tau)(x\li\ot(x\lii\trl b)-(b\trr x\li)\ot x\lii)=0$,

\item[(F9)]$x\li\ot\si(x\lii, y)+x\boo\ot(x\bi\ppl y)+xy\boo\ot y\bi+(x\trl y\qi)\ot y\qii)$\\
$-y\li\ot\si(y\lii, x)-y\boo\ot(y\bi\ppl x)-yx\boo\ot x\bi-(y\trl x\qi)\ot x\qii)$\\
$=\tau\Big(x\boi\ot x\boo y+x\qi\ot(x\qii\trr y)+\si(x, y\li)\ot y\lii+(x\ppr y\boi)\ot y\boo$\\
$-y\boi\ot y\boo x-y\qi\ot(y\qii\trr x)-\si(y, x\li)\ot x\lii-(y\ppr x\boi)\ot x\boo\Big)$,

\item[(F10)]$(\id-\tau)(x\boi\ot(x\boo\ppr b)+x\qi\ot x\qii b+(x\ppr b\li)\ot b\lii$\\
$=(\id-\tau)(b\li\ot (b\lii \ppl x)+(b\ppl x\boo)\ot x\bi+bx\qi\ot x\qii)$,

\item[(F11)]$(\id-\tau)(x\boi\ot \si(x\boo, y)+x\qi\ot(x\qii\ppl y)+\si(x, y\boo)\ot y\bi+(x\ppr y\qi)\ot y\qii)$\\
$=(\id-\tau)(y\boi\ot \si(y\boo, x)+y\qi\ot(y\qii\ppl x)+\si(y, x\boo)\ot x\bi+(y\ppr x\qi)\ot x\qii)$,
\end{enumerate}
\begin{enumerate}
\item[(F12)] $\Delta_{V}(xy)+\tau\Delta_{V}(y x)$\\
$=y x_{2}\otimes x_{1} + y_{2}\otimes y_{1}x + x_{1} y \otimes x_{2}+y_{1} \otimes x y_{2}$\\
$ +(y\trl x_{[1]}) \otimes x_{[0]}+y_{[0]} \otimes(y_{[-1]} \trr x )+(x_{[-1]}\trr y) \otimes x_{[0]}+y_{[0]} \otimes (x\trl y_{[1]}), $

\item[(F13)]
$(\id-\tau)(x\li\ot x\lii y+xy\li\ot y\lii-y\li\ot y\lii x-yx\li\ot x\lii)$\\
$+(\id-\tau)(x\boo\ot(x\bi\trr y)+(x\trl y\boi)\ot y\boo-y\boo\ot(y\bi\trr x)-(y\trl x\boi)\ot x\boo)=0$.
\end{enumerate}
Conversely, any    anti-flexible bialgebra structure on $E$ with the canonical injection map $i: A\to E$ both an anti-flexible algebra homomorphism and an anti-flexible   coalgebra homomorphism is of this form.
\end{theorem}
Note that in this case, $(A, \, \cdot, \, \Delta_A)$ is an anti-flexible  sub-bialgebra of $E=A^{}_{\sigma}\# {}^{Q}_{}\, V$ and $(V, \, \cdot, \, \Delta_V)$ is a  braided   anti-flexible bialgebra.
Denote the set of all     anti-flexible bialgebraic extending datum of type (II) by $\Omega^{(II)}({A}, \,  V)$.

In the above two cases, we find that  the braided    anti-flexible bialgebra $V$ play a special role in the extending problem of   anti-flexible bialgebra $A$.
Note that $A^{P}_{}\# {}^{}_{\theta}\, V$ and $A^{}_{\sigma}\# {}^{Q}_{}\, V$ are all    anti-flexible bialgebra structures on $E$.
Conversely,  any    anti-flexible bialgebra extending system $E$ of ${A}$  through $V$ is isomorphic to such two types.
Now from Theorem \ref{thm-41} and Theorem \ref{thm-42},  we obtain the main result of in this section,
which solve the extending problem for anti-flexible bialgebra.


\begin{theorem}\label{bim1}
Let $({A}, \cdot, \Delta_A)$ be an anti-flexible bialgebra, $E$ a vector space containing ${A}$ as a subspace and $V$ be a complement of ${A}$ in $E$.
Denote by
$$\mathcal{HLB}(V,{A}):=\Omega^{(I)}({A},V)\sqcup\Omega^{(II)}({A},V)/\equiv.$$
Then the map
\begin{eqnarray*}
&&\Upsilon: \mathcal{HLB}(V,{A})\rightarrow BExtd(E,{A}),\\
&&\overline{\Omega^{(I)}({A},V)}\mapsto A^{P}_{}\# {}^{}_{\theta}\, V,\quad   \overline{\Omega^{(II)}({A},V)}\mapsto A^{}_{\sigma}\# {}^{Q}_{}\, V
\end{eqnarray*}
is bijective, where $\overline{\Omega^{(I)}({A}, V)}$ and $\overline{\Omega^{(II)}({A}, V)}$ are the equivalence classes of $\Omega^{(I)}({A}, V)$ and $\Omega^{(II)}({A}, V)$ under $\equiv$ respectively.
\end{theorem}

\vskip7pt
\footnotesize{
\noindent Tao Zhang\\
College of Mathematics and Information Science,\\
Henan Normal University, Xinxiang 453007, P. R. China;\\
 E-mail address: \texttt{{zhangtao@htu.edu.cn}}

\vskip7pt
\footnotesize{
\noindent Hui-jun Yao\\
College of Mathematics and Information Science,\\
Henan Normal University, Xinxiang 453007, PR China};\\
 E-mail address: \texttt{{yhjdyxa@126.com}}


\begin{thebibliography}{99}

\bibitem{AM1}
A. L. Agore, G. Militaru,  \emph{Extending structures I: the level of groups}, {Algebr. Represent. Theory} { 17} (2014), 831--848.

\bibitem{AM2}
A. L. Agore, G. Militaru, \emph{Extending structures II: the quantum version},  J. Algebra { 336} (2011), 321--341.


\bibitem{AM3}
A.L. Agore, G. Militaru, \emph{Extending structures for Lie algebras}, Monatsh. fur Mathematik 174 (2014), 169--193.

\bibitem{AM4}
A. L. Agore, G. Militaru, \emph{Unified products for Leibniz algebras. Applications}, Linear Algebra Appl. { 439} (2013), 2609--2633.


\bibitem{AM5}
A.L. Agore, G. Militaru, \emph{The global extension problem, crossed products and co-flag noncommutative Poisson algebras}, J. Algebra 426 (2015), 1--31.

\bibitem{AM6}
A. L. Agore, G. Militaru, \emph{Extending structures, Galois groups and supersolvable associative algebras}, Monatsh. Math. {181} (2016), 1--33.


%
%

\bibitem{AO}
C. T. Anderson and D. L. Outcalt, \emph{On simple anti-flexible rings}, J. Algebra 10 (1968), 310--320.


\bibitem{Bai10} C. Bai,  \emph{Double constructions of Frobenius algebras, Connes cocycles and their duality}, J. Noncommut. Geom. 4 (2010), 475--530.





\bibitem{BD99}Y.~Bespalov and B.~Drabant,  \emph{Cross product bialgebras - Part I},  J. Algebra { 219} (1999), 466--505.

\bibitem{BD01}Y.~Bespalov and B.~Drabant,  \emph{Cross product bialgebras - Part II}, J. Algebra { 240} (2001), 445--504.


\bibitem{DBH} M. L. Dassoundo, C. Bai and M. N. Hounkonnou, \emph{Anti-flexible bialgebras}, J. Algebra Appl.  {21} (2022), 2250212

\bibitem{D1} M. L. Dassoundo, \emph{Pre-anti-flexible bialgebras}, Comm. Algebra 49(9) (2021), 4050--4078.


\bibitem{D2} M. L. Dassoundo,
\emph{Relative (pre-)anti-flexible algebras and associated algebraic structures}, Quasigroups and Related Systems 30 (2022), 31--46.



\bibitem{Dr86} V. G.~Drinfel'd,    \emph{Quantum groups}. In ``Proceedings International Congress of Mathematicians, August 3-11, 1986, Berkeley, CA" pp. 798--820, Amer. Math. Soc.  Providence, RI, 1987.

%
%
%
%

\bibitem{Hong}
Y. Hong,  \emph{Extending structures for Lie bialgebras}, arXiv:2108.05586.


\bibitem{K}
F. Kosier, \emph{On a class of nonflexible algebras}, Trans. Amer. Math. Soc. 102 (1962), 299--318.



\bibitem{Ma90a}
S.~Majid,   \emph{Matched pairs of   groups associated to solutions of the Yang-Baxter equations}, Pacific J. Math.  { 141} (1990), 311--332.

\bibitem{Ma95} S.~Majid,    \emph{Founditions of Quantum Groups}, Cambridge: Cambridge University  Press, 1995.

\bibitem{Ma00} S.~Majid,   \emph{Braided Lie bialgebras}, Pacific J. Math.  { 192} (2000), 329--356.

\bibitem{Mas00} A.~Masuoka,   \emph{Extensions of Hopf algebras and  Lie bialgebras}, Trans. Amer. Math. Soc.   { 352} (2000), 3837--3879.


\bibitem{Ra85} D.~E.~Radford,  \emph{The structure of Hopf algebras with a projection},  J. Algebra  { 92} (1985), 322--347.

\bibitem{R}
D. Rodabaugh, \emph{On antiflexible algebras}, Trans. Amer. Math. Soc. 169 (1972), 219--235.


\bibitem{So96}  Y.~Sommerh\"{a}user,  \emph{Kac-Moody algebras}, Presented at the Ring Theory Conference, Miskolc, Hungary, Preprint, 1996.

\bibitem{S02}  Y.~Sommerh\"{a}user, \emph{Yetter-Drinfel'd Hopf algebras over groups of prime order}, Lect. Notes Math.  1789, Springer, Berlin, 2002

\bibitem{Zh99} S.~C.~Zhang,  H.~X.~Chen,   \emph{The double bicrossproducts in braided tensor categories},  Comm. Algebra, { 29}(1) (2001), 31--66.

\bibitem{Z1}T. Zhang, \emph{Double cross biproduct and bi-cycle bicrossproduct Lie bialgebras}, J. Gen. Lie Theory Appl. { 4} (2010), S090602.

\bibitem{Z2}T. Zhang, \emph{Unified products for braided Lie bialgebras with applications},  J. Lie Theory { 32}(3) (2022), 671--696.

\bibitem{Z3}T. Zhang, \emph{Extending structures for 3-Lie algebras},   Comm. Algebra  { 50}(4) (2022), 1469--1497.

\bibitem{Z4} T. Zhang, \emph{Extending structures for infinitesimal bialgebras},  arXiv:2112.11977v1.

\bibitem{ZCY}
J. Zhao, L. Chen, L. Yuan, \emph{Extending structures of  Lie conformal superalgebras}, Comm. Algebra 47(4) (2019), 1541--1555.



\end{thebibliography}
\end{document}